\documentclass[12pt]{amsart}
\usepackage{amssymb}
\usepackage{amsmath}
\usepackage{amsthm}
\usepackage{times}
\usepackage{graphicx}

\usepackage[curve,all]{xy}
 \xyoption{curve}
 \xyoption{line}
\xyoption{arc}

\usepackage{color}
\usepackage{amsthm}
\newtheorem{theorem}{Theorem}[section]
\newtheorem{lemma}[theorem]{Lemma}

\newtheorem{proposition}[theorem]{Proposition}
\newtheorem{prop}[theorem]{Proposition}

\theoremstyle{definition}

\newtheorem{example}[theorem]{Example}

\theoremstyle{remark}
\newtheorem{remark}[theorem]{Remark}
\newtheorem{conjecture}[theorem]{Conjecture}

\numberwithin{equation}{section}

\newcommand{\calA}{\mathcal{A}}
\newcommand{\calB}{\mathcal{B}}

\newcommand{\calO}{\mathcal{O}}

\newcommand{\calS}{\mathcal{S}}

\newcommand{\la}{\langle}
\newcommand{\ra}{\rangle}

\def\Aut{{\text{Aut}}}

\def\PGL{{\text{PGL}}}

\def\Num{{\rm{Num}}}

\def\deg{{\text{deg}}}

\def\Num{{\text{Num}}}

\def\II{{\text{II}}}

\begin{document}
\title [Enriques surfaces] {Classification of Enriques surfaces covered by the supersingular $K3$ surface with Artin invariant 1 in characteristic 2}

\author{Shigeyuki Kond\=o}
\address{Graduate School of Mathematics, Nagoya University, Nagoya,
464-8602, Japan}
\email{kondo@math.nagoya-u.ac.jp}
\thanks{Research of the author is partially supported by
Grant-in-Aid for Scientific Research (S) No. 15H05738.}

\begin{abstract}
We classify Enriques surfaces covered by the supersingular $K3$ surface with the Artin invariant 1 in
characteristic 2.  There are exactly three types of such Enriques surfaces.  
%We also give a notion of $R$-invariants for Enriques surfaces whose canonical coverings are supersingular $K3$ surfaces with rational double points.
\end{abstract}

\maketitle

\rightline{Dedicated to Toshiyuki Katsura on  his 70th birthday}
%\bigskip

\section{Introduction}\label{sec1}

In this paper we work over an algebraically closed field $k$ of characteristic 2.  
It is known that there exist complex Enriques surfaces whose covering $K3$ surfaces are mutually isomorphic.  In \cite{Ko} the author gave two non isomorphic Enriques surfaces whose covering $K3$ surfaces
are the same Kummer surface, and
Ohashi \cite{Oh1}, \cite{Oh2} investigated such Enriques surfaces by using the theory of periods of Enriques surfaces.  In particular he showed that the number of isomorphism classes of such Enriques surfaces with a given $K3$ surface as their coverings is finite.

On the other hand, Enriques surfaces in characteristic 2 have a different phenomenon.
Ekedahl, Hyland and Shepherd-Barron \cite{EHS} showed that 
the moduli space of Enriques surfaces whose canonical covers are supersingular $K3$ surfaces with 
twelve nodes is an open set of a ${\bf P}^1$-bundle over the moduli space of lattice polarized (called $N$-marked in \cite{EHS}) supersingular $K3$ surfaces.  
Here ${\bf P}^1$ parametrizes derivations on such a $K3$ surface. Note that the moduli space of Enriques surfaces (resp. supersingular $K3$ surfaces) has dimension 10 (resp. dimension 9).  Thus the number of isomorphism classes
of Enriques surfaces with a given supersingular $K3$ surface as their canonical coverings is infinite in general.  

The purpose of this paper is to give an explicit description of Enriques surfaces whose canonical cover
is the most special supersingular $K3$ surface.
Recall that the moduli space of supersingular $K3$ surfaces is stratified by the Artin invariant $\sigma$ $(1\leq \sigma \leq 10)$ such
that $K3$ surfaces with Artin invariant $\sigma$ form a $(\sigma -1)$-dimensional family
(Artin \cite{A}).  Moreover
supersingular $K3$ surfaces with Artin invariant 1 are unique up to isomorphisms (Ogus \cite[Cor. 7.14]{Og}   for $p > 2$, Rudakov-Shafarevich \cite[\S 4]{RS2} for $p=2$).
%Also there are several examples of Enriques surfaces together with their canonical coverings \cite{KK1}, \cite{KK2}, \cite{KKM}.
In this paper we determine all Enriques surfaces
covered by the supersingular $K3$ surface with Artin invariant 1.
The following is the main theorem of this paper.  For precise statements, see Theorems \ref{Main1}, \ref{Main2}, Remark \ref{EHSB-unique}.

\begin{theorem}\label{main}
There exist exactly three types of Enriques surfaces such that the minimal resolutions of the canonical double covers of these Enriques surfaces are the supersingular $K3$ surface with Artin invariant $1$.
Each type of them forms a $1$-dimensional family.  
\end{theorem}

\begin{remark}
The canonical covers of Enriques surfaces of two types have twelve nodes and the one of the remaining type has a rational double point of type $D_4$ and eight nodes.  
\end{remark}

\begin{remark}
Every Enriques surface $X$ in the three families has
a finite number of $(-2)$-divisors such that the reflection group generated by reflections associated
with these $(-2)$-divisors is of finite index in the orthogonal group of ${\rm Num}(X)$, where
${\rm Num}(X)$ is the N\'eron-Severi group of $X$ modulo
the torsion subgroup.
\end{remark}

We give examples of three types in Theorem \ref{main} explicitly.  The first one called of type $\rm MI$ was given in Katsura and Kondo \cite{KK1} which is a 1-dimensional family of classical and supersingular Enriques surfaces.
Their canonical covers have twelve nodes.  Each member $X$ of this family 
contains 30 nodal curves (non-singular rational curves) and 10 non-effective $(-2)$-divisors whose 
dual graph satisfies a condition for the finiteness of the index of the corresponding 
reflection group in the orthogonal group ${\rm O}({\rm Num}(X))$ (see Proposition \ref{Vinberg}).  The second type appeared as one of Enriques surfaces with finite automorphism group (called Type VII in Katsura and Kondo \cite{KK2}, Katsura, Kondo and Martin \cite{KKM}).  It is also a 1-dimensional family of classical and supersingular Enriques surfaces whose canonical covers have twelve nodes.  Each member of the family contains exactly 20 nodal curves whose dual graph satisfies the same condition. The third and final one called of type $\rm MII$ is new and will be given in section \ref{sec3}.  It is a 1-dimensional family of classical Enriques surfaces whose canonical covers have a rational double point of type $D_4$ and eight nodes.  It contains 28 nodal curves and 12 non-effective $(-2)$-divisors whose dual graph satisfies the same condition.

%In the paper \cite{N}, Nikulin introduced an important invariant, called the $R$-invariant, for complex Enriques surfaces which mesures how nodal curves sit on an Enriques surface.  This notion is easily extended to the case that the canonical covering is \'etale, that is, ${\rm char}(k)=p>2$ or singular Enriques surfaces in characteristic 2.  We extend this notion to the case that the canonical covers are supersingular $K3$ surfaces with rational double points in characteristic 2 and caluculate $R$-invariants of three examples (Proposition \ref{R-inv-examples}).

To prove Theorem \ref{main} we use the classification of all elliptic fibrations on
the supersingular $K3$ surfaces with Artin invariant 1 and their uniqueness due to Elkies and Sch\"utt
\cite{ES} in an essential way.  We fix one of the possible elliptic fibrations on such an Enriques 
surface $X$ and a bi-section of
this fibration.  Then we can see that there exists a unique type of Enriques surface $X'$ 
among three types
such that it has an elliptic fibration of the same type and a bi-section of given type.
By lifting the fibration to the canonical cover and applying the uniqueness of such elliptic fibration,
we can see that $X$ has the same configuration of nodal curves as that of $X'$.
Finally, together with a result
by Ekedahl, Hyland and Shepherd-Barron \cite[Theorem 3.21]{EHS}, these examples give all Enriques surfaces
covered by the supersingular $K3$ surface with the Artin invariant 1 (see Remark \ref{EHSB-unique}).

The plan of this paper is as follows.  In section \ref{sec2}, we recall the known results
on Enriques surfaces and supersingular $K3$ surfaces.
In section \ref{sec3}, we recall and give three examples of Enriques surfaces covered by the supersingular 
$K3$ surface with Artin invariant 1.
Section \ref{sec4} is devoted to possible singularities of the canonical covers of 
Enriques surfaces of desired type and possible types of
elliptic fibrations on them. In section \ref{sec5} we determine possibilities of bi-sections of each special elliptic fibrations, and 
in section \ref{sec6} we will state and give a proof of the main theorems \ref{Main1}, \ref{Main2}.
%, and in section \ref{sec6}, we discuss $R$-invariants of supersingular and classical Enriques surfaces whose canonical covers have only rational double points. 

\medskip
\noindent
{\bf Acknowledgement.} The author thanks Toshiyuki Katsura for valuable conversations and 
Shigeru Mukai for informing the author the dual graph of $(-2)$-vectors in the Example of
type ${\rm MII}$. The author thanks Matthias Sch\"utt and Yuya Matsumoto for reading the manuscript
and for suggesting misprints and useful comments.  In particular Matsumoto pointed out that
Enriques surfaces of type ${\rm MII}$ form a 1-dimensional family.  The author thanks the referee for careful reading of the manuscript, for pointing out errors and for useful suggestions.

\section{Preliminaries}\label{sec2}

A lattice is a free abelian group $L$ of finite rank equipped with 
a non-degenerate symmetric integral bilinear form $\langle . , . \rangle : L \times L \to {\bf Z}$. 
For a lattice $L$ and an integer $m$, we denote by $L(m)$ the free ${\bf Z}$-module $L$ 
with the bilinear form obtained from the bilinear form of $L$ by multiplication by $m$. 
The signature of a lattice is the signature of the real vector space $L\otimes {\bf R}$ 
equipped with the symmetric bilinear form extended from the one on $L$ by linearity. A lattice is called even if 
$\langle x, x\rangle \in 2{\bf Z}$ 
for all $x\in L$. 
We denote by $U$ the even unimodular lattice of signature $(1,1)$, 
and by $A_m, \ D_n$ or $\ E_k$ the even negative definite lattice defined by
the Cartan matrix of type $A_m, \ D_n$ or $\ E_k$ respectively.    
We denote by $L\oplus M$ the orthogonal direct sum of lattices $L$ and $M$, 
and by $L^{\oplus m}$ the orthogonal direct sum of $m$-copies of $L$.
Let ${\rm O}(L)$ be the orthogonal group of $L$, that is, the group of isomorphisms of $L$ preserving the bilinear form.
%For an even lattice $L$, we denote $L^*$ the dual of $L$ and define a map 
%\begin{equation}\label{discr-quadratic}
%q_L: L^*/L \to {\bf Q}/2{\bf Z}, \quad q_L(x \ {\rm mod}\ L) = x^2 \ {\rm mod}\ 2{\bf Z},
%\end{equation}
%which is called the discriminant quadratic form of $L$.  
%For more details we refer the reader to Nikulin \cite{N}.

Let $k$ be an algebraically closed field of characteristic $p > 0$,
and let $S$ be a nonsingular complete algebraic surface defined over $k$.
We denote by $K_{S}$ the canonical divisor of $S$.
A rational vector field $D$ on $S$ is said to be $p$-closed if there exists
a rational function $f$ on $S$ such that $D^p = fD$. 
A vector field $D$ is of additive type (resp. of multiplicative type) if 
$D^p=0$ (resp. $D^p=D$).
Let $\{U_{i} = {\rm Spec} A_{i}\}$ be an affine open covering of $S$. We set 
$A_{i}^{D} = \{\alpha \in A_{i} \mid D(\alpha) = 0\}$. 
The affine varieties $\{U_{i}^{D} = {\rm Spec} A_{i}^{D}\}$ glue together to 
define a normal quotient surface $S^{D}$.

Now, we  assume that $D$ is $p$-closed. Then,
the natural morphism $\pi : S \longrightarrow S^D$ is a purely
inseparable morphism of degree $p$. 
If the affine open covering $\{U_{i}\}$ of $S$ is fine enough, then
taking local coordinates $x_{i}, y_{i}$
on $U_{i}$, we see that there exist $g_{i}, h_{i}\in A_{i}$ and 
a rational function $f_{i}$
such that the divisors defined by $g_{i} = 0$ and by $h_{i} = 0$ have no common components,
and such that
$$
 D = f_{i}\left(g_{i}\frac{\partial}{\partial x_{i}} + h_{i}\frac{\partial}{\partial y_{i}}\right)
\quad \mbox{on}~U_{i}.
$$
By Rudakov and Shafarevich \cite[Section 1]{RS}, divisors $(f_{i})$ on $U_{i}$
glue to a global divisor $(D)$ on $S$, and the zero-cycle defined
by the ideal $(g_{i}, h_{i})$ on $U_{i}$ gives rise to a well-defined global zero cycle 
$\langle D \rangle $ on $S$. A point contained in the support of
$\langle D \rangle $ is called an isolated singular point of $D$.
If $D$ has no isolated singular point, $D$ is said to be divisorial.
Rudakov and Shafarevich \cite[Theorem 1, Corollary]{RS}
showed that $S^D$ is nonsingular
if $\langle D \rangle  = 0$, i.e., $D$ is divisorial.
When $S^D$ is nonsingular,
they also showed a canonical divisor formula
\begin{equation}\label{canonical}
K_{S} \sim \pi^{*}K_{S^D} + (p - 1)(D),
\end{equation}
where $\sim$ means linear equivalence.
As for the Euler number $c_{2}(S)$ of $S$, we have a formula
\begin{equation}\label{euler}
c_{2}(S) = \deg \langle D \rangle  - \langle K_{S}, (D)\rangle - (D)^2
\end{equation}
(cf. Katsura and Takeda \cite[Proposition 2.1]{KT}). 
Now we consider an irreducible curve $C$ on $S$ and we set $C' = \pi (C)$.
Take an affine open set $U_{i}$ as above such that $C \cap U_{i}$ is non-empty.
The curve $C$ is said to be integral with respect to the vector field $D$
if $g_{i}\frac{\partial}{\partial x_{i}} + h_{i}\frac{\partial}{\partial y_{i}}$
is tangent to $C$ at a general point of $C \cap U_{i}$. Then, Rudakov-Shafarevich
\cite[Proposition 1]{RS} showed the following proposition:

\begin{prop}\label{insep}

\noindent
\begin{itemize}
\item[$(1)$]  
If $C$ is integral, then $C = \pi^*(C')$ and $C^2 = pC'^2$.
\item[$(2)$]   
If $C$ is not integral, then $pC = \pi^*(C')$ and $pC^2 = C'^2$.
\end{itemize}
\end{prop}

In any characteristic ${\rm char}(k)=p \geq 0$, an algebraic surface with numerically trivial
canonical divisor is called an Enriques surface if the second Betti
number is equal to 10. In case of $p=2$,
Enriques surfaces $X$ are divided into three classes
(for details, see Bombieri and Mumford \cite[Section 3]{BM2}):
\begin{itemize}
\item[$(1)$] $K_{X}$ is not linearly equivalent to zero 
and $2K_{X}\sim 0$.  Such an Enriques surface is called a classical Enriques surface.
\item[$(2)$] $K_{X} \sim 0$, ${\rm H}^{1}(X, {\calO}_{X}) \cong k$
and the Frobenius map acts on  ${\rm H}^{1}(X, {\calO}_X)$ bijectively.
Such an Enriques surface is called a singular Enriques surface.
\item[$(3)$] $K_{X} \sim 0$, ${\rm H}^{1}(X, {\calO}_{X}) \cong k$
and the Frobenius map is the zero map on  ${\rm H}^{1}(X, {\calO}_X)$.
Such an Enriques surface is called a supersingular Enriques surface.
\end{itemize}

It is known that the canonical cover of any singular Enriques surface is not supersingular.  Moreover
it is an ordinary $K3$ surface (e.g. Katsura and Kondo \cite[Theorem A.1]{KK2}).  
Recently Liedtke \cite{L} showed that 
the moduli space of
Enriques surfaces with a polarization of degree $4$ has two $10$-dimensional irreducible components.  A general point of one component (resp. the other component) corresponds to a singular (resp. classical) Enriques surface, and the intersection of the two components parametrizes supersingular Enriques surfaces.

Now assume that   
$X$ is a classical or supersingular Enriques surface and $\bar{\pi} : \bar{Y}\to X$ the canonical cover.
In this case there exists a regular 1-form $\eta$ on $X$.  A point $P\in \bar{Y}$ is a singular point if and only if $\eta$ vanishes at $\bar{\pi}(P)$ (Bombieri-Mumford \cite[p.221]{BM2}).  
Since $c_2(X)=12$, $\eta$ has 12 zeros generically.  Thus in case of classical or supersingular Enriques surfaces, they have always
a singularity.
We call the points of zeros of $\eta$ canonical points of $X$.
If $\bar{Y}$ has only rational double points, then the minimal resolution of singularities is a supersingular $K3$ surface, and 
it is a rational surface otherwise (Cossec and Dolgachev \cite[Theorem 1.3.1]{CD}).

We call a nonsingular rational curve on an Enriques surface or $K3$ surface a nodal curve.
If $C$ is a nodal curve, then $C^2=-2$. 

Let $X$ be a supersingular or classical Enriques surface.  Let $\bar{\pi} : \bar{Y} \to X$ be the canonical
cover.  Assume that $\bar{Y}$ has only rational double points.  Let $\rho : Y \to \bar{Y}$ be the minimal resolution. 

\begin{lemma}\label{nodal}{\rm (Ekedahl and Shepherd-Barron \cite[Definition-Lemma 0.8]{EkS})}
Let $E$ be a nodal curve on $X$ and denote by $\tilde{E}$ the irreducible curve on
$Y$ mapping surjectively to $E$.  Then 
$\tilde{E}$ is a nodal curve, the degree of the map $\bar{\pi} \circ \rho : \tilde{E}\to E$ is one, and
two points (including infinitely near points) on $E$ are blow-ups during the minimal resolution.
If two nodal curves $E_1$ and $E_2$ on $X$ meet transversally at one point, then their strict transforms do not meet on $Y$.
\end{lemma}

Here we recall the theory of supersingular $K3$ surfaces (Artin \cite{A}).  In any algebraically closed field $k$ 
in characteristic $p >0$,
a $K3$ surface $Y$ defined over $k$ is called supersingular if the Picard number of $Y$ is 22.
Let $Y$ be a supersingular $K3$ surface.  Denote by ${\rm NS}(Y)$ the N\'eron-Severi group of $Y$ and
by ${\rm NS}(Y)^*$ the dual of ${\rm NS}(Y)$.
Then ${\rm NS}(Y)$ is an even lattice of signature $(1,21)$ such that ${\rm NS}(Y)^*/{\rm NS}(Y)$ is isomorphic to a $p$-elementary abelian group $({\bf Z}/p{\bf Z})^{2\sigma}$ where $\sigma$ is called the Artin invariant of $Y$ and satisfies $1\leq \sigma \leq 10$.   The supersingular $K3$ surfaces with the Artin invariant $\sigma$ form a $(\sigma -1)$-dimensional family.  Moreover supersingular $K3$ surfaces with $\sigma =1$ are unique up to isomorphisms (Ogus \cite[Cor. 7.14]{Og} for $p > 2$, Rudakov-Shafarevich \cite[\S 4]{RS2} for $p=2$).  A concrete example of the supersingular $K3$ 
surface in characteristic 2 with Artin invariant 1 is given as follows (see Dolgachev-Kondo \cite{DK}):
Let ${\bf P}^2({\bf F}_4)$ be the projective plane over the finite field ${\bf F}_4$.
It contains 21 points and 21 lines, and each line contains five points and each point is contained in
five lines.  Let $Z$ be the inseparable double cover of ${\bf P}^2$ defined by
$$t^2 = x^4yz + y^4xz + z^4xy$$
where $(x,y,z)$ are homogeneous coordinates of ${\bf P}^2$.  The partial derivatives of this equation are
$$y^4z+z^4y, \quad x^4z+z^4x, \quad x^4y+ y^4x$$
all of which vanish exactly at  21 ${\bf F}_4$-rational points of ${\bf P}^2$.  Thus $Z$ has 21 rational double points of type $A_1$.  Let $Y$ be the minimal resolution of $Z$ which is a $K3$ surface. 
Obviously $Y$ contains the disjoint union of the 21 nodal
curves which are exceptional curves of the resolution.  On the other hand, the pullbacks of the 21 lines in ${\bf P}^2({\bf F}_4)$ are 21 disjoint nodal curves.
Thus we have two sets $\calA$ and $\calB$ of disjoint 21 nodal curves such that each member in one set meets exactly five members in the other set at one point transversally.   These 42 nodal curves generate 
the N\'eron-Severi lattice ${\rm NS}(Y)$ which has rank 22 and discriminant $-2^2$.  Thus $Y$ is the supersingular
$K3$ surface with the Artin invariant 1.

Now we recall some facts on elliptic fibrations on Enriques surfaces and the supersingular $K3$ surface
with the Artin invariant 1.

\begin{prop}\label{multi-fiber}{\rm (Cossec and Dolgachev \cite[Theorems 5.7.5,\ 5.7.6]{CD})}

Let $f : X \to {\bf P}^1$ be an elliptic fibration on an Enriques surface $X$ in 
characteristic $2$.  Then the following hold.

\begin{itemize}
\item[$(1)$]  If $X$ is classical, then $f$ has two tame multiple fibers with multiplicity $2$, each is
either an ordinary elliptic curve or a singular fiber of additive type.
\item[$(2)$]  If $X$ is singular, then $f$ has one wild multiple  
fiber with multiplicity $2$ which is an ordinary elliptic curve or a singular fiber of multiplicative type.
\item[$(3)$]  If $X$ is supersingular, then $f$ has one wild multiple fiber with multiplicity $2$ 
which is a
supersingular elliptic curve or a singular fiber of additive type.
\end{itemize}
\end{prop}

We use Kodaira's notation for singular fibers of an elliptic fibration: 
$${\rm I}_n,\ {\rm I}_n^*,\ {\rm II},\ {\rm II}^*,\ {\rm III},\ {\rm III}^*,\ {\rm IV},\  {\rm IV}^*.$$
If an elliptic fibration on $X$ has a multiple fiber, for example, of type ${\rm III}$, 
then we call it a fiber of type $2{\rm III}$.

Let $X$ be an Enriques surface and $f:X\to {\bf P}^1$ an elliptic fibration. 
Since $f$ has a multiple fiber, it has no sections.  If $f$ admits a bi-section $s$ isomorphic to
a nodal curve, then $f$ is called special and $s$ is called a special bi-section.
%It is known that if $S$ contains a nodal curve, then there exists a special genus one fibration.
The following result is due to Cossec \cite{Cossec} in which he assumed the characteristic $p\not=2$, but the assertion for $p=2$ holds, too.

\begin{prop}\label{Cossec}{\rm (Lang \cite[$\II$, Theorem A3]{Lang})}
Assume that an Enriques surface $X$ contains a nodal curve.  
Then there exists a special genus one fibration on $X$.
\end{prop}

Elliptic fibrations (genus one fibrations more generally) on the supersingular $K3$ surfaces with Artin invariant 1 have been classified (Kondo and Shimada \cite{KS}, Elkies and Sch\"utt \cite{ES}). Moreover Elkies and Sch\"utt proved the following theorem.

\begin{theorem}\label{ElkiesSchutt}{\rm (Elkies and Sch\"utt \cite[Theorem 1, Theorem 2, Proposition 9]{ES})}
Let $Y$ be the supersingular $K3$ surface with Artin invariant $1$ over an algebraically closed field $k$
in characteristic $2$.  Then $Y$ admits exactly $18$ genus $1$ fibrations.  More precisely, for each genus $1$ fibration, there is exactly one model over $k$ up to isomorphisms.  Moreover any genus $1$ fibration has a section.
\end{theorem} 

\noindent
It is enough to consider only elliptic fibrations in our situation by the following Proposition.

\begin{prop}{\rm (Cossec and Dolgachev \cite{CD}, Proposition 5.7.3)}
Let $X$ be an Enriques surface.  Assume that its canonical cover has only rational double points.
Then $X$ does not admit quasi-elliptic fibrations.
\end{prop}

Among 18 genus 1 fibrations, there are 8 elliptic fibrations.
The following is the list of elliptic fibrations.

\begin{theorem}\label{ElkiesSchutt2}{\rm (Elkies and Sch\"utt \cite[Theorem 1]{ES}, Kondo and Shimada \cite[Theorem 4.7]{KS})}
There are exactly the following eight types of singular fibers of elliptic fibrations on $Y$.
$$({\rm I}_6, {\rm I}_6, {\rm I}_6, {\rm I}_6), \ ({\rm I}_8, {\rm I}_8, {\rm I}_1^*), \ ({\rm I}_{10}, {\rm I}_{10}, {\rm I}_2, {\rm I}_2),\ ({\rm I}_{12}, {\rm I}_3^*),\ $$
$$({\rm I}_{12}, {\rm I}_4, {\rm IV}^*),\ ({\rm IV}^*, {\rm IV}^*, {\rm IV}^*),\ ({\rm I}_{16}, {\rm I}_1^*), \ ({\rm I}_{18}, {\rm I}_2, {\rm I}_2, {\rm I}_2).$$
\end{theorem}

Finally we recall the theory of reflection groups in hyperbolic spaces.
First we consider the case of Enriques surfaces.
Let $X$ be an Enriques surface and let $\Num(X)$ be the quotient of the N\'eron-Severi group 
of $X$ by the torsion subgroup.  Then $\Num(X)$ together with the intersection product is
an even unimodular lattice of signature $(1,9)$ (Illusie \cite{Ill}). 
We denote by ${\rm O}(\Num(X))$ the orthogonal group of $\Num(X)$. The set 
$$\{ x \in \Num(X)\otimes {\bf R} \ : \ \langle x, x \rangle > 0\}$$ 
has two connected components.
Denote by $P(X)$ the connected component containing an ample class of $X$.  
For $\delta \in \Num(X)$ with $\delta^2=-2$, we define
an isometry $s_{\delta}$ of $\Num(X)$ by
$$s_{\delta}(x) = x + \langle x, \delta\rangle \delta, \quad x \in \Num(X).$$ 
The isometry $s_{\delta}$ is called the reflection associated with $\delta$.
Let $W(X)$ be the subgroup of
${\rm O}(\Num(X))$ generated by reflections associated with all nodal  
curves on $X$.  Then $P(X)$ is divided into chambers 
each of which is a fundamental domain with respect to
the action of $W(X)$ on $P(X)$.  We remark that the automorphism group ${\rm Aut}(X)$ is finite if the index $[{\rm O}(\Num(X)) : W(X)]$ is finite (Dolgachev \cite[Proposition 3.2]{D}).

Now, we recall Vinberg's result
which guarantees that a group generated by a finite number of reflections is
of finite index in the orthogonal group. 
Let $L$ be an even lattice of signature $(1,n)$.
Let $\Delta$ be a finite set of $(-2)$-vectors in $L$.
Let $\Gamma$ be the graph of $\Delta$, that is,
$\Delta$ is the set of vertices of $\Gamma$ and two vertices $\delta$ and $\delta'$ are joined 
by $m$-tuple lines if $\langle \delta, \delta'\rangle=m$.
We assume that the cone
$$K(\Gamma) = \{ x \in L\otimes {\bf R} \ : \ \langle x, \delta_i \rangle \geq 0, \ \delta_i \in \Delta\}$$
is a strictly convex cone. Such $\Gamma$ is called non-degenerate.
A connected parabolic subdiagram $\Gamma'$ in $\Gamma$ is a  Dynkin diagram of 
type $\tilde{A}_m$, $\tilde{D}_n$ or $\tilde{E}_k$ (see Vinberg \cite[p. 345, Table 2]{V}).  
If the number of vertices of $\Gamma'$ is $r+1$, then $r$ is called the rank of $\Gamma'$.  
A disjoint union of connected parabolic subdiagrams is called a parabolic subdiagram of $\Gamma$.  
We denote by $\tilde{K_1}\oplus \tilde{K_2}$ a parabolic subdiagram which is a disjoint union of two
connected parabolic subdiagrams of type $\tilde{K_1}$ and $\tilde{K_2}$, where
$K_i$ is $A_m$, $D_n$ or $E_k$. The rank of a parabolic subdiagram is the sum of 
the rank of its connected components.  Note that the dual graph of reducible fibers 
of an elliptic fibration gives a parabolic subdiagram.  
For example, a singular fiber of type ${\rm III}$, ${\rm IV}$ or ${\rm I}_{n+1}$ 
defines a parabolic subdiagram of type $\tilde{A}_1$, $\tilde{A}_2$ or 
$\tilde{A}_n$ respectively.  
We denote by $W(\Gamma)$ the subgroup of ${\rm O}(L)$ 
generated by reflections associated with $\delta \in \Gamma$.

\begin{prop}\label{Vinberg}{\rm (Vinberg \cite[Theorem 2.3]{V})}
Let $\Delta$ be a set of $(-2)$-vectors in an even lattice $L$ of signature $(1,n)$ 
and let $\Gamma$ be the graph of $\Delta$.
Assume that $\Delta$ is a finite set, $\Gamma$ is non-degenerate and $\Gamma$ contains 
no $m$-tuple lines with $m \geq 3$.  Then $W(\Gamma)$ is of finite index in ${\rm O}(L)$ 
if and only if every connected parabolic subdiagram of $\Gamma$ is a connected component of some
parabolic subdiagram in $\Gamma$ of rank $n-1$ {\rm (}= the maximal one{\rm )}.
\end{prop} 

\begin{remark}\label{Vinbergremark}
Note that $\Gamma$ as in the above proposition is automatically non-degenerate if it contains the components of the reducible fibers of an extremal genus one fibration on an Enriques surface and a special bi-section of this fibration. Indeed, these nodal curves generate 
${\rm Num}(X)\otimes {\bf Q}$ and hence $K(\Gamma)$ is strictly convex.
\end{remark}

Let $L$ be an even lattice isomorphic to the N\'eron-Severi lattice of the supersingular $K3$ surface 
$Y$ in characteristic 2 with the Artin invariant 1.  Then $L$ has the signature $(1,21)$ and the  discriminant $-2^2$.  
In this case the reflection subgroup generated by
reflections associated with all $(-2)$-vectors is not of finite index in ${\rm O}(L)$.  However the subgroup generated by all reflections (not only $(-2)$-reflections, but also $(-4)$-reflections in ${\rm O}(L)$) is of finite index in ${\rm O}(L)$. Such lattice is called reflective. There exist reflective lattices of
signature $(1,n)$ only if $1\leq n \leq 19$ or $n=21$ (Esselmann \cite{Esselmann}).
Moreover the lattice $L$ is the only known example of reflective lattices in rank 22 due to 
Borcherds \cite{Bo}. The automorphism group ${\rm Aut}(Y)$ is infinite, that is, the ample cone has infinitely many facets.  Here a facet means a face of codimension 1.  
On the other hand, there exists a finite polyhedron in the ample cone which has
42 facets defined by 42 $(-2)$-vectors and 168 facets defined by 168 $(-4)$-vectors.
The 42 $(-2)$-vectors correspond to 42 nodal curves in $\calA$ and $\calB$ on $Y$.  
The 168 $(-4)$-vectors correspond to 
\begin{equation}\label{168}
2h - (E_1 + \cdots + E_6)
\end{equation}
where $h$ is the pullback of the class of a line on ${\bf P}^2$ under the map $Y \to Z \to {\bf P}^2$ and $E_1, \ldots, E_6$ are nodal curves over six points in general position on ${\bf P}^2({\bf F}_4)$.
Here a set of six points on ${\bf P}^2({\bf F}_4)$ is called general if no three points are collinear.
There are exactly 168 sets of six points in general position.  
Each of these 42 $(-2)$- and 168 $(-4)$-vectors defines a reflection in ${\rm O}(L)$.
The finite polyhedron is a fundamental domain of the group generated by all reflections associated with 42 $(-2)$- and 168 $(-4)$-vectors.  The reflections associated with 
168 $(-4)$-vectors are realized by automorphisms of $Y$.  Thus we can give a generator of
${\rm Aut}(Y)$ (Dolgachev-Kondo \cite{DK}).

\section{Examples}\label{sec3}

\subsection{Enriques surfaces of type ${\rm MI}$}\label{A}

This example was given in Katsura and the author \cite{KK1}.  We recall it briefly.
Let
$$x_1^2x_2 + x_1x_2^2 + x_0^3 + sx_0(x_1^2+x_1x_2+x_2^2)=0$$
be a pencil of cubics on ${\bf P}^2$ with a parameter $s$.  The base points of the pencil are nine 
${\bf F}_4$-rational points. There are exactly four members ($s^3=1$ and $s=\infty$) in the pencil which consist of three lines on ${\bf P}^2({\bf F}_4)$.  
By blowing-up the nine base points we have a rational elliptic
surface with four singular fibers of type ${\rm I}_3$ and with nine sections.  Recall that there 
exist exactly 5 lines in ${\bf P}^2({\bf F}_4)$ passing a point 
in ${\bf P}^2({\bf F}_4)$.  This implies that there are nine bi-sections of the elliptic fibration
passing a singular point of singular fibers.
Now consider the Frobenius base change $t^2=s$ of the pencil
$$x_1^2x_2 + x_1x_2^2 + x_0^3 + t^2x_0(x_1^2+x_1x_2+x_2^2)=0$$
which has 12 rational double points of type $A_1$ over the singularities of singular fibers of type 
${\rm I}_3$.  
We will use its affine model 
$$y^2 + y + x^3 + t^2x(y^2 + y +1)=0.$$
By resolution of singularities, we have an elliptic fibration 
$$g: Y \to {\bf P}^1$$
which has four singular fibers of type ${\rm I}_6$ and 18 sections.
Note that the 9 base points and the 12 singular points of the singular fibers of type ${\rm I}_3$ of the cubic pencil are exactly the 21 ${\bf F}_4$-rational points on ${\bf P}^2$.  Thus
$Y$ is birational to the inseparable double covering of ${\bf P}^2$ given in \S 2.  
Hence $Y$ is the supersingular $K3$ surface with Artin invariant 1.
Note that $Y$ contains 42 nodal curves which are 24 components of singular fibers of $g$
and 18 sections.  These 42 nodal curves correspond to
21 lines and 21 points on ${\bf P}^2({\bf F}_4)$.  Now consider a rational derivation defined by
\begin{equation}\label{derMI}
D_{a,b}= 
{1\over (t-1)}\left((t-1)(t + a)(t+b)\frac{\partial}{\partial t} + (1 + t^2x)\frac{\partial}{\partial x}\right)
\end{equation}
where $a, b \in k, \ a+b=ab,\ a^3\not=1$.
Then $D_{a,b}^2 = abD_{a,b}$, that is, $D_{a,b}$ is 2-closed.  
It is known that $D_{a,b}$ is divisorial, and hence the quotient surface $Y^{D_{a,b}}$ is nonsingular.
Moreover the integral nodal curves with respect to $D_{a,b}$ are the disjoint union of twelve nodal curves which are components of four singular fibers of type ${\rm I}_6$.
By blowing-down twelve $(-1)$-curves on $Y^{D_{a,b}}$ which are the images of integral nodal curves, we have 
an Enriques surface $X_{a,b}$.  The fibration $g: Y\to {\bf P}^1$ induces 
an elliptic fibration $f : X \to {\bf P}^1$ which has four singular fibers of type ${\rm I}_3$ and 
18 special bi-sections.
Thus there are 30 nodal curves on $X_{a,b}$.  On the other hand, among the 168 divisors given in (\ref{168}),
there are exactly ten divisors which are orthogonal to all twelve integral nodal curves.  The images of
these ten $(-4)$-divisors descend to ten $(-2)$-divisors on $X_{a,b}$.
%The dual graph of the 30 nodal curves and the 10 $(-2)$-divisors satisfies the .

\begin{theorem}\label{thmMI}{\rm (Katsura and Kondo \cite[Theorems 4.8, 7.5]{KK1})} There exists a $1$-dimensional family $\{X_{a,b}\}$ of classical and supersingular
Enriques surfaces whose canonical covers
$\bar{Y}_{a,b}$ have twelve nodes.  Here $a, b \in k,\ a+b=ab,\ a^3\not=1$.  
The minimal resolution of each $\bar{Y}_{a,b}$  is the supersingular 
$K3$ surface $Y$ with the Artin invariant $1$.  
If $a=0$, then $X_{a,b}$ is supersingular, and otherwise classical.
Each $X_{a,b}$ contains $30$ nodal curves and $10$ non-effective $(-2)$-classes which satisfy  the condition in Proposition {\rm \ref{Vinberg}}.  In particular the reflection subgroup generated by
reflections associated with these $40$ $(-2)$-vectors is of finite index in ${\rm O}({\rm Num}(X_{a,b}))$.
\end{theorem}

We mention an another detail of the 30 nodal curves.
There exist twelve canonical points on $X_{a,b}$ which are the images of twelve integral curves.
Each nodal curve passes through two canonical points.  
Recall that there are 42 nodal curves in $\calA$ and $\calB$.
We have decompositions
$$\calA = \calA_0 \cup \calA_1, \quad \calB = \calB_0 \cup \calB_1$$
where both $\calA_0$ and $\calB_0$ consist of six integral curves.  We denote by $\bar{\calA}_0$ and
$\bar{\calB}_0$ the sets of six canonical points on $X_{a,b}$ which are the 
images of $\calA_0$ and $\calB_0$, respectively.
In the following Figure \ref{A5},
the six black nodes denote the six canonical points in $\bar{\calA}_0$ or in $\bar{\calB}_0$, and 
the 15 lines denote the 15 nodal curves passing through two canonical points from the six canonical points.
Thus we conclude that the 30 nodal curves are divided into two sets of 15 nodal curves whose incidence relation is given in Figure \ref{A5}.  
Nodal curves in Figure \ref{A5} meet only at canonical points.
%Any two nodal curves in Figure \ref{A5} meet only at one canonical point. 
Each member in a set is tangent to exactly three members in another set.

\begin{figure}[!htb]
 \begin{center}
  \includegraphics[width=40mm]{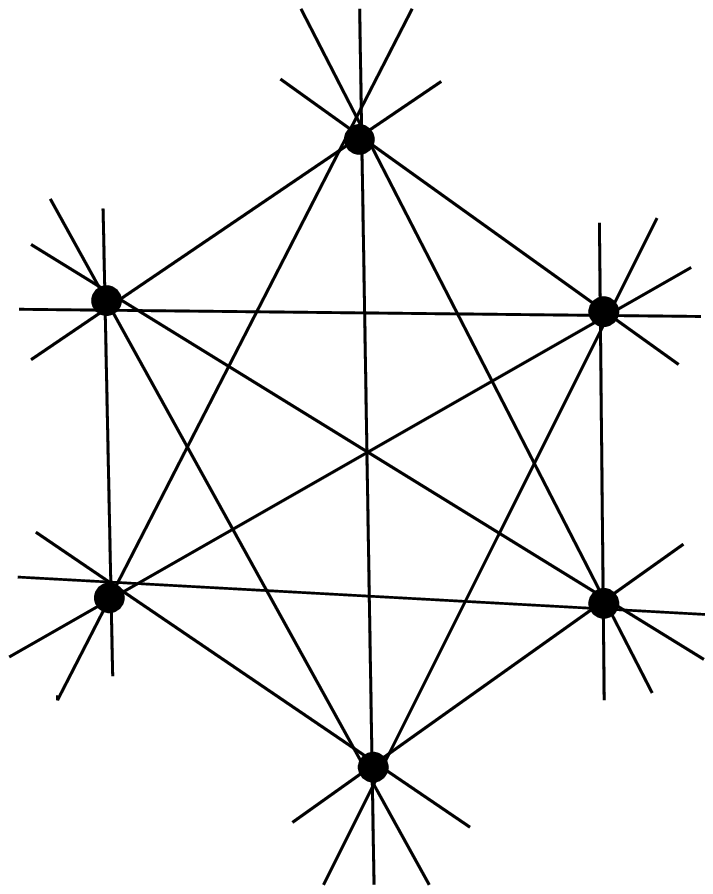}
 \end{center}
 \caption{}
 \label{A5}
\end{figure}

The set of elliptic fibrations on $Y$ up to ${\rm Aut}(Y)$ bijectively corresponds to
the set of primitive isotropic vectors in ${\rm NS}(Y)$ contained in the closure of the finite polyhedron
defined by 42 nodal curves and 168 $(-2)$-curves.  It follows that any elliptic 
fibration on $X_{a,b}$ is isomorphic to one of fibrations corresponding to primitive isotropic vectors in ${\rm Num}(X_{a,b})$ contained in the closure of the finite polyhedron
defined by $40$ $(-2)$-vectors mentioned in Theorem \ref{thmMI}. Thus we have the following Proposition. 

\begin{proposition}\label{MIfibrations}{\rm (Katsura and Kondo \cite[Lemma 7.2 and the subsequent arguments]{KK1})}
There exist exactly four types of elliptic fibrations on $X_{a,b}$ as follows$:$
$$({\rm I}_5, {\rm I}_5, {\rm I}_1, {\rm I}_1),\ ({\rm I}_6,  2{\rm IV}, {\rm I}_2), \ 
({\rm I}_4, {\rm I}_4, 2{\rm III}),\ ({\rm I}_3, {\rm I}_3, {\rm I}_3, {\rm I}_3).$$ 
In each case there are exactly twelve singular points of fibers which are canonical points of $X_{a,b}$, that is, the images of twelve integral curves.  All elliptic fibrations are special.
\end{proposition}

\noindent

\begin{lemma}\label{bi-section1} 
Let $f:X_{a,b}\to {\bf P}^1$ be an elliptic fibration.
\noindent
\begin{itemize}
\item[$(1)$] In case that $f$ is of type $({\rm I}_5, {\rm I}_5, {\rm I}_1, {\rm I}_1)$, any special bi-section passes through a singular point of a fiber of type ${\rm I}_5$ and that of a fiber of type ${\rm I}_1$, and is tangent to each other singular fiber at a simple point.
\item[$(2)$] In case of $({\rm I}_6,  2{\rm IV}, {\rm I}_2)$, any special bi-section passes through 
a singular point of 
the fiber of type ${\rm I}_2$ and is tangent to the fiber of type ${\rm I}_6$ at a simple point.  
There are four
canonical points on the multiple fiber of type ${\rm IV}$ one of them is the singular point of the fiber and others are simple points on each component.  Any special bi-section passes through a canonical point on the multiple fiber. 
\item[$(3)$] In case of type $({\rm I}_4, {\rm I}_4, 2{\rm III})$, any special bi-section passes through 
a singular point of 
a fiber of type ${\rm I}_4$ and is tangent to the other fiber of type ${\rm I}_4$ at a simple point.
There are four canonical points on the multiple fiber of type ${\rm III}$.  Each component of the multiple fiber contains two canonical points both of which are simple points of the fiber.  Any special bi-section passes 
through a canonical point on the multiple fiber. 
\item[$(4)$] In case of type $({\rm I}_3, {\rm I}_3, {\rm I}_3, {\rm I}_3)$, any special bi-section passes through a singular point of two fibers of type ${\rm I}_3$, and is tangent to  
each other singular fiber at a simple point.
\end{itemize}
\end{lemma}
\begin{proof}
In cases (1), (2), (4), the elliptic fibration $g$ on $Y$ induced from $f$ has a section and
its Mordell-Weil group is a torsion group.  Any special bi-sections of $f$ is 
one of 30 nodal curves mentioned in Theorem \ref{thmMI}.  Thus we directly prove the assertion.

In case (3), $g$ has singular fibers of type $({\rm I}_8, {\rm I}_8, {\rm I}_1^*)$.  Twelve
canonical points are the singular points of two fibers of type ${\rm I}_4$ and
four points on the singular fiber of type ${\rm III}$.  The last four points are the images of
four simple components of the fiber of $g$ of type ${\rm I}_1^*$.  Since the pullback of any special bi-section of $f$ is a section of $g$, it 
passes exactly one canonical point on the fiber of type ${\rm III}$.  Since any nodal curve passes two
canonical points (Lemma \ref{nodal}), we have the assertion.
\end{proof}

\begin{remark}
Over the complex numbers, S. Mukai obtained an Enriques surface which contains 30 nodal curves 
with the same dual graph as the above example.  The name "of type MI" comes from this fact.  The canonical cover of the Mukai's example is the intersection of three quadrics given by the equations:
$$x^2 - (1+\sqrt{3})yz = u^2 - (1-\sqrt{3})vw,$$
$$y^2 -  (1+\sqrt{3})xz = v^2 - (1-\sqrt{3})uw,$$
$$z^2 -  (1+\sqrt{3})xy = w^2 - (1-\sqrt{3})uv.$$
See Mukai and Ohashi \cite[Remark 2.7]{MO}.
\end{remark}

\subsection{Enriques surfaces of type ${\rm VII}$ }\label{B}

This type has appeared in the classification of Enriques surfaces with finite
automorphism group (Katsura and Kondo \cite{KK2}, Katsura, Kondo and Martin \cite{KKM}).
  
We start with a rational elliptic fibration defined by
$$y^2 + sxy + y = x^3 + x^2 +s$$
which has two singular fibers of type ${\rm I}_5$ over $s=1,\infty$ and two singular fibers of type
${\rm I}_1$ over $s=\omega, \omega^2 \ (\omega^3=1, \omega\not=1)$.
Taking the Frobenius base change $s=t^2$, we have an elliptic fibration $g : Y \to {\bf P}^1$ defined by
\begin{equation}\label{eqVII}
y^2 + t^2xy + y = x^3 + x^2 +t^2.
\end{equation}
The fibration $g$ has two singular fibers of type ${\rm I}_{10}$ over $t=1, \infty$ and
two singular fibers of type ${\rm I}_2$ over $t=\omega, \omega^2$.  And $g$ has 10 sections.
One can prove that $Y$ is the supersingular $K3$ surface with the Artin invariant 1.

Now consider a rational derivation defined by
$$D_{a,b}= 
{1\over (t-1)}\left((t-1)(t -a)(t-b)\frac{\partial}{\partial t} + (1 + t^2x)\frac{\partial}{\partial x}\right)
$$
where $a, b \in k$, $a+b=ab$ and $a^3\not=1$ (this derivation is the same as in the case of type MI given
in (\ref{derMI}).  However the equations of these surfaces are different).
Then $D_{a,b}^2 = abD_{a,b}$, that is, $D_{a,b}$ is 2-closed.  
It is known that $D_{a,b}$ is divisorial, and hence the quotient surface $Y^{D_{a,b}}$ is smooth.
Moreover the integral nodal curves with respect to $D_{a,b}$ are the disjoint union of twelve nodal curves which are components of the singular fibers of type ${\rm I}_{10}$ and of type ${\rm I}_2$.
By blowing-down the twelve $(-1)$-curves on $Y^{D_{a,b}}$ which are the images of integral nodal curves, we have an Enriques surface $X_{a,b}$.  The fibration $g$ induces an
elliptic fibration $f : X_{a,b} \to {\bf P}^1$ which has two singular fibers of type ${\rm I}_5$
and two singular fibers of type ${\rm I}_1$.  The ten sections of $g$ give ten bi-sections of $f$.  Thus there are 20 nodal curves on $X_{a,b}$ whose dual graph coincides with that of
the Enriques surface of type VII defined over ${\bf C}$ in Kondo \cite{Ko}.
The following Figure \ref{fano} is a part of the 20 nodal curves.  Each line denotes a nodal curve and the 10 black circles are a part of the 12 canonical points.  The remaining five nodal curves pass through
the remaining two canonical points. 
The dual graph of the 20 nodal curves satisfies the condition in Proposition \ref{Vinberg}.  In particular the reflection subgroup generated by
reflections associated with these $20$ $(-2)$-vectors is of finite index in ${\rm O}({\rm Num}(X_{a,b}))$.

\begin{figure}[!htb]
 \begin{center}
  \includegraphics[width=40mm]{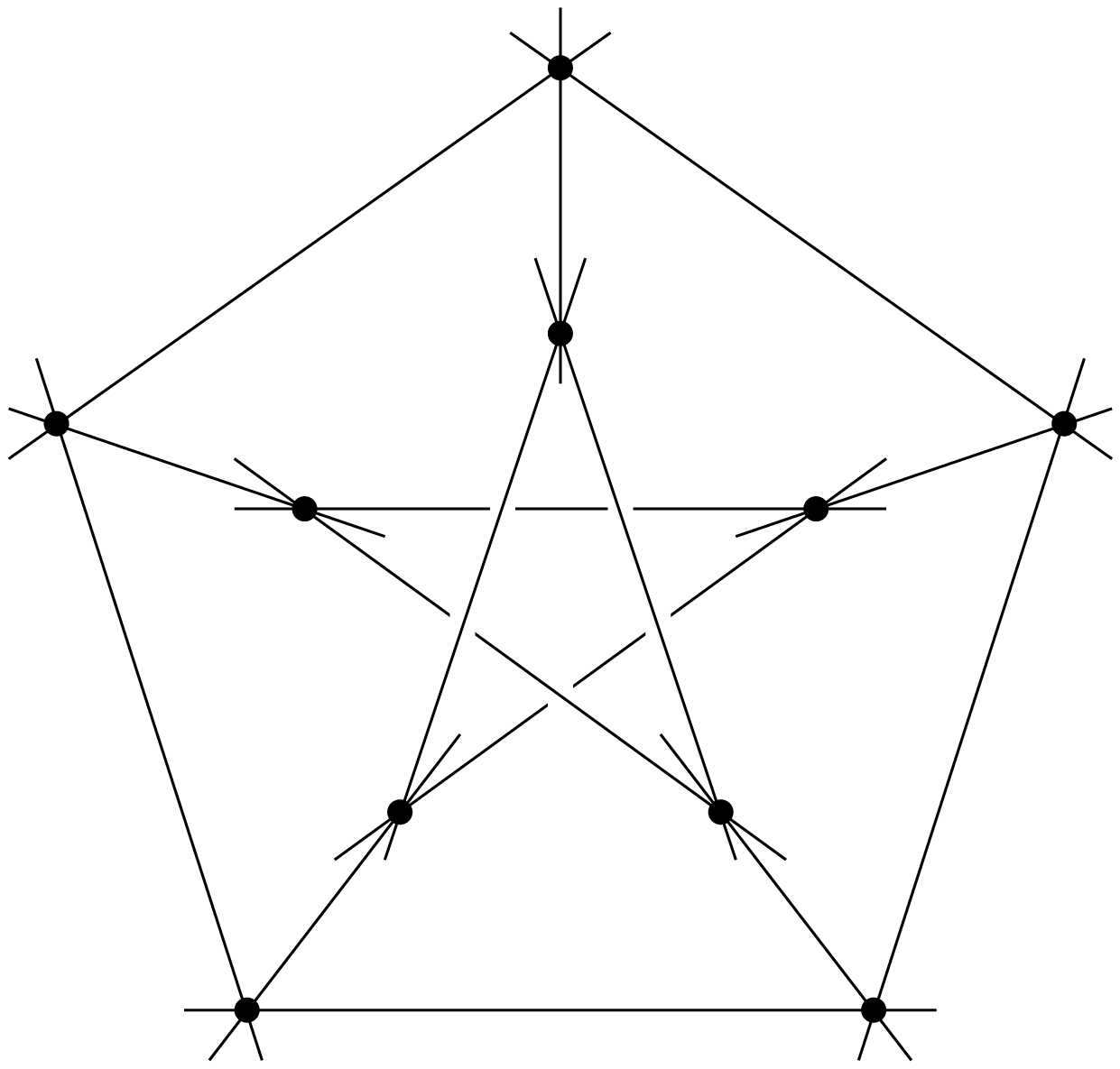}
 \end{center}
 \caption{}
 \label{fano}
\end{figure}

\begin{theorem}{\rm (Katsura and Kondo \cite[Theorems 3.15, 3.19]{KK2})} 
There exists a $1$-dimensional family $\{X_{a,b}\}$ of classical and supersingular 
Enriques surfaces whose canonical covers
$\bar{Y}_{a,b}$ have twelve nodes.  Here $a, b \in k,\ a+b=ab,\ a^3\not=1$.
The minimal resolution of each $\bar{Y}_{a,b}$ is the supersingular 
$K3$ surface $Y$ with Artin invariant $1$.  The surface $X$ contains exactly $20$ nodal curves and
the automorphism group ${\rm Aut}(X_{a,b})$ is isomorphic to the symmetric group $\mathfrak{S}_5$ of degree $5$.
\end{theorem}

\begin{proposition}{\rm (Katsura and Kondo \cite[Figure 2 and the subsequent arguments]{KK2})}
There exist exactly four types of elliptic fibrations on $X_{a,b}$ as follows$:$
$$({\rm I}_5, {\rm I}_5, {\rm I}_1, {\rm I}_1),\ ({\rm I}_6, 2{\rm IV}, {\rm I}_2),\ ({\rm I}_9, {\rm I}_1, {\rm I}_1, {\rm I}_1),\ ({\rm I}_8, 2{\rm III}).$$ 
All elliptic fibrations are special.
\end{proposition}

\begin{lemma}\label{bi-section2} 
Let $f:X_{a,b}\to {\bf P}^1$ be an elliptic fibration.
\noindent
\begin{itemize}
\item[$(1)$] In case that $f$ is of type $({\rm I}_5, {\rm I}_5, {\rm I}_1, {\rm I}_1)$, the following two cases occur.
Either a special bi-section passes through a singular point of two fibers of type ${\rm I}_5$ and is
tangent to the other two singular fibers at a simple point, or it passes through a singular point of two fibers of type ${\rm I}_1$ and is tangent to the other two singular fibers at a simple point.
\item[$(2)$] In case of $({\rm I}_6,  2{\rm IV}, {\rm I}_2)$, any special bi-section passes 
through a singular point of 
the fiber of type ${\rm I}_6$ and is tangent to the fiber of type ${\rm I}_2$ at a simple point.  
There are four canonical points on the multiple fiber of type ${\rm IV}$ one of them is the singular point of the fiber and others are simple points on each component.  Any special bi-section passes through 
a canonical point on the multiple fiber. 
\item[$(3)$] In case of type $({\rm I}_9, {\rm I}_1, {\rm I}_1, {\rm I}_1)$, the following two cases occur.
Either a special bi-section passes through a singular point of 
a fiber of type ${\rm I}_9$ and a singular point of a fiber of type ${\rm I}_1$, and is tangent to
the other two fibers at a simple point, or it passes through the singular point of two fibers of 
type ${\rm I}_1$ and is tangent to the other two fibers at a simple point.
\item[$(4)$] In case of type $({\rm I}_8, 2{\rm III})$, any special bi-section passes through 
a singular point of the fiber of type ${\rm I}_8$.  There are four canonical points on the multiple fiber of type ${\rm III}$.  Each component of the multiple fiber contains two canonical points both of which are simple points of the fiber.  Any special bi-section passes through a canonical point on the multiple fiber. 
\end{itemize}
\end{lemma}
\begin{proof}
The surface $X_{a,b}$ has exactly 20 nodal curves.  
Recall that the elliptic fibration $f:X_{a,b}\to {\bf P}^1$ has two singular fibers of
type ${\rm I}_5$ and ten bi-sections, and the above 20 nodal curves are exactly the components of
singular fibers and bi-sections of $f$.  
The fibration $g: Y \to {\bf P}^1$ inducing $f$ is defined by
the equation (\ref{eqVII}).  We know the intersection relations between the fibers and sections of $g$ explicitly (see Katsura and Kondo \cite[\S 3]{KK2}).  Thus we have the assertions.
\end{proof}

\subsection{Enriques surfaces of type ${\rm MII}$}\label{C}
Consider a line $\ell$ in ${\bf P}^2({\bf F}_4)$ and denote by $p_1,..., p_5$ the five ${\bf F}_4$-rational points on $\ell$.  For $i=1,2$, let $\ell_{ij}$ $(j=1,...,4)$ be the four lines 
in ${\bf P}^2({\bf F}_4)$ through $p_i$ except $\ell$
(see Figure \ref{A7-1}).

\begin{figure}[!htb]
 \begin{center}
  \includegraphics[width=40mm]{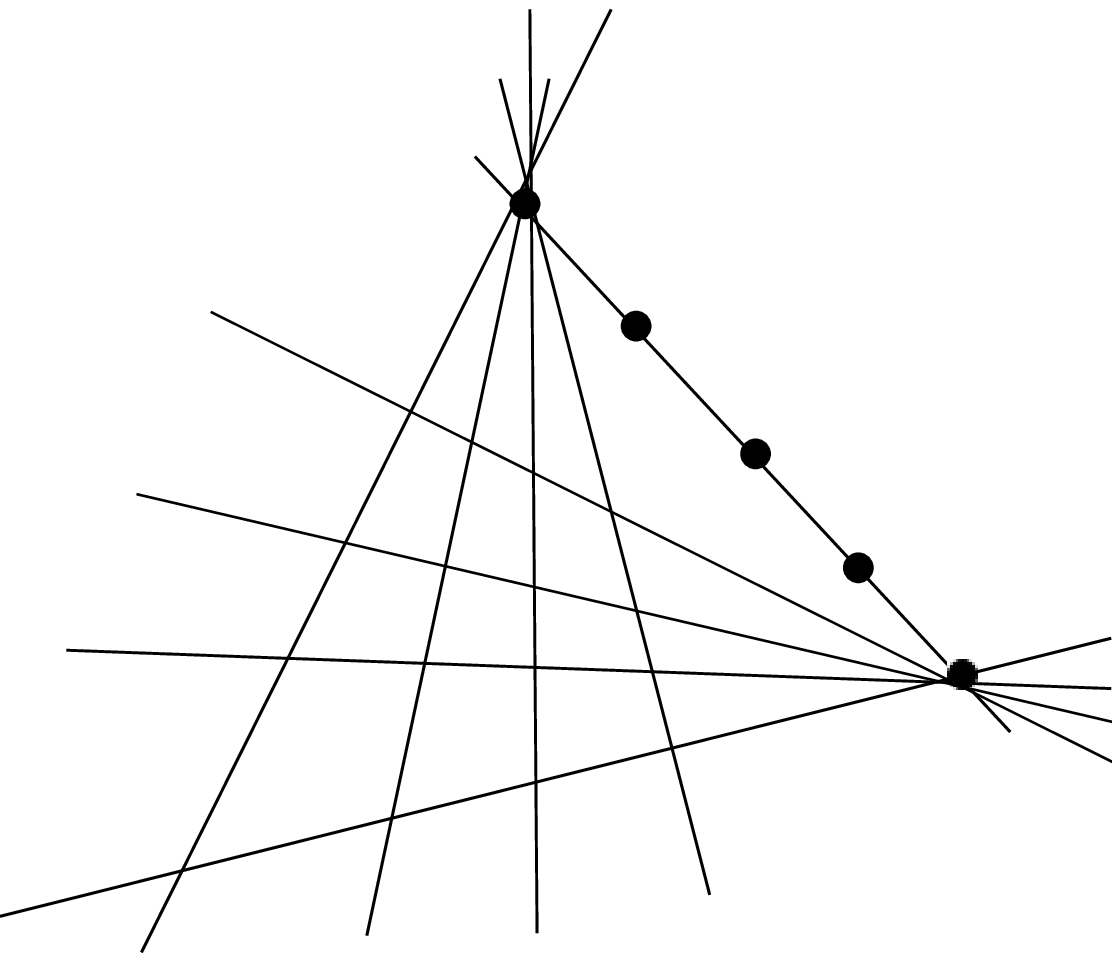}
 \end{center}
 \caption{}
 \label{A7-1}
\end{figure}

Let $Y$ be the supersingular $K3$ surface with Artin invariant 1.  Recall that $Y$ contains 42 nodal curves which are the pullbacks of the 21 lines on ${\bf P}^2({\bf F}_4)$ and the 21 exceptional curves
over the 21 points on ${\bf P}^2({\bf F}_4)$.
Let $L, L_{ij}$ be the proper transforms of $\ell, \ell_{ij}$ on $Y$.  Also denote by $E_i$ the exceptional curve over the point $p_i$ $(i=1,...,5)$.
Let $\bar{Y}$ be the surface obtained by contracting $L_{ij}$ $(i=1,2, j=1,\ldots, 4)$, $L, E_3, E_4, E_5$ which has eight rational double points of type $A_1$ and one rational double point of type $D_4$.  We shall give classical Enriques surfaces $X=X_{a,b}$ whose canonical covers are $\bar{Y}$.  
The surfaces $X$ contain 28 nodal curves as in Figure \ref{A7-2}.

\begin{figure}[!htb]
 \begin{center}
  \includegraphics[width=50mm]{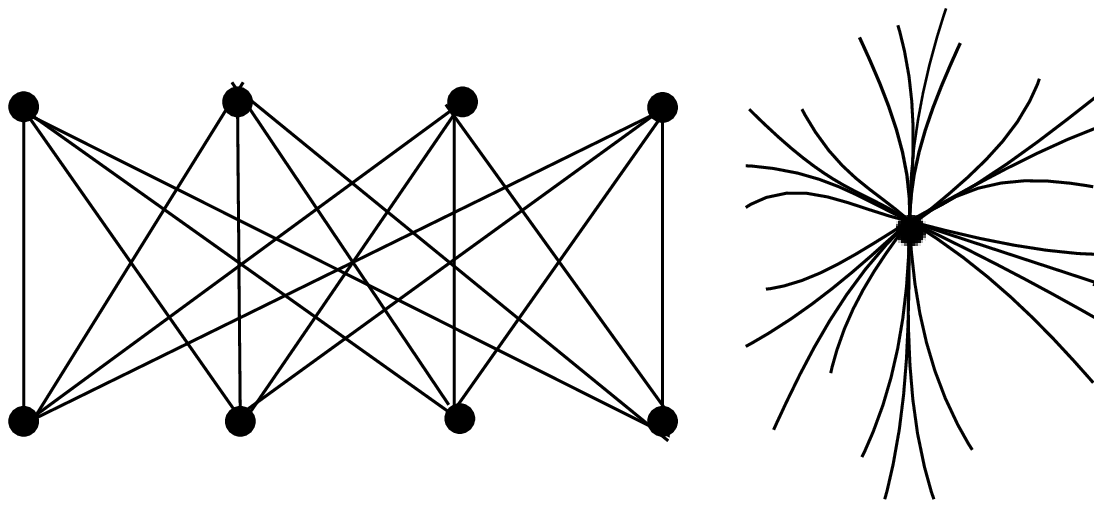}
 \end{center}
 \caption{}
 \label{A7-2}
\end{figure}

\noindent
Among 28 nodal curves in Figure \ref{A7-2}, sixteen of them are
the images of the sixteen exceptional curves $E_{ij}$ on $Y$ over the sixteen intersection points of $\ell_{1i}$ and $\ell_{2j}$, and twelve of them are the images of the twelve lines on 
${\bf P}^2({\bf F}_4)$ through $p_3$, $p_4$ or $p_5$.  The 16 straight lines on the left hand side in
Figure \ref{A7-2} denote these 16 nodal curves and the eight black circles denote eight canonical points which are the images of eight rational double points of type $A_1$.  The 16 nodal curves meet  
only at the 8 canonical points.  On the right hand side in Figure
\ref{A7-2} twelve curves denote the twelve nodal curves.  All these twelve curves
pass through the black circle corresponding to the canonical point which is the image of the 
rational double point of type $D_4$.  
These twelve nodal curves are divided into three groups each of which consist of four nodal curves
tangent each other.

To construct $X$ we use an elliptic fibration $g:Y\to {\bf P}^1$ defined by
\begin{equation}\label{ell-fibMII}
y^2 + xy + t^2(t+1)^2y = x^3 + t^2(t+1)^2x^2
\end{equation}
which is the Frobenius base change $s=t^2$ of a rational elliptic surface
defined by
$$y^2 + xy + s(s+1)y = x^3 + s(s+1)x^2.$$
The fibration $g$ 
has two singular fibers of type ${\rm I}_8$ over $t=0,1$ and one of type ${\rm I}_1^*$ over $t=\infty$.  This elliptic fibration is realized by the linear system
$$|L_{11} + E_{11} +L_{21}+E_{21}+L_{12}+E_{22}+L_{22}+E_{12}|.$$
The other singular fibers are given by the divisors
$$L_{13} + E_{33} +L_{23}+E_{43}+L_{14}+E_{44}+L_{24}+E_{34}$$
and
$$E_3 + E_4 + 2(L + E_5) + F_1 + F_2$$
where $F_1, F_2$ are proper transforms of some lines passing through $p_5$.
Thus the linear system defines an elliptic fibration on $Y$ with singular fibers of type 
$({\rm I}_8,{\rm I}_8,{\rm I}_1^*)$.
By the uniqueness of elliptic fibrations on $Y$ (Theorem \ref{ElkiesSchutt}), we may assume
that the fibration defined by the linear system is the one $g:Y\to {\bf P}^1$ given by (\ref{ell-fibMII}).

Now consider a rational derivation $D_{a,b}$ on $Y$ induced by
$$
{1\over abt(t+1)}\left(t(t + 1)(at+1)(bt+1)\frac{\partial}{\partial t} + (x + t^2(t+1)^2)\frac{\partial}{\partial x}\right)
$$
where $a, b \in k^*,\ a+b=ab$ (the author gave a derivation (the case $a=\omega, b=\omega^2$) and later Yuya Matsumoto pointed out the existence of derivations of this type).
Obviously $D_{a,b}$ has poles of order 1 along the fibers over the points $t=0, 1$ and
the fibers over the points defined by $t=0,1, 1/a, 1/b$ are integral with respect to $D_{a,b}$.
We resolve the singularities of the surface defined by the equation (\ref{ell-fibMII}).
Then we calculate the divisorial part of the induced derivation, denoted by the same symbol $D_{a,b}$, 
on $Y$ and determine integral curves on fibers.  These are elementary, but long calculations.
Thus one can prove the following Lemma.

\begin{lemma}\label{MIIDer}

\noindent
\begin{itemize}
\item[$(1)$] $D_{a,b}^2 = D_{a,b}$, namely, $D_{a,b}$ is $2$-closed and of multiplicative type. 
\item[$(2)$] 
On the surface $Y$, the divisorial part of $D_{a,b}$ is given by
$$
(D_{a,b}) = - (L_{11} + L_{12} + L_{21} +  L_{22} + L_{13} +  L_{14}
+  L_{23} + L_{24} +  2(L + E_{3} + E_{4} +  E_{5}))
$$
and $(D_{a,b})^2 = -24$.
\item[$(3)$] The integral curves with respect to
$D_{a,b}$ in the fibers of $g : Y \longrightarrow {\bf P}^{1}$ are the following$:$
two smooth fibers over the points $t=1/a, 1/b$ and
$$L_{11},\ L_{12},\ L_{21},\ L_{22},\ L_{13},\ L_{14},\ L_{23},\ L_{24},\ E_{3},\ E_{4},\ E_{5}.$$
\end{itemize}
\end{lemma}

\begin{lemma}\label{divisorial}
The derivation $D_{a,b}$ is divisorial.
\end{lemma}
\begin{proof}
It follows from the formula (\ref{euler}) and Lemma \ref{MIIDer}, (2) that
$$24 = c_2(Y) = {\rm deg}(\la D_{a,b}\ra) - K_Y\cdot (D_{a,b}) - (D_{a,b})^2 = {\rm deg}(\la D_{a,b}\ra) + 24.$$
Hence ${\rm deg}(\la D_{a,b}\ra) =0$ and the assertion follows.
\end{proof}

It follows from Lemma \ref{divisorial} that the quotient surface $Y^{D_{a,b}}$ is nonsingular.
We denote by $\pi:Y\to Y^{D_{a,b}}$ the quotient map.
By using Lemma \ref{insep} we see that $Y^{D_{a,b}}$ has 
11 exceptional curves of the first kind which are the images of the rational integral curves stated in
Lemma \ref{MIIDer}, (3).
By contracting these curves and then contracting $\pi(L)$ we get a smooth surface $\phi: Y^{D_{a,b}}\to X_{a,b}$.
It follows from the formula (\ref{canonical}) that 
$$0 = K_Y = \pi^*(K_{Y^{D_{a,b}}}) + (D_{a,b}).$$
On the other hand, by construction, we have
$K_{Y^{D_{a,b}}} = \phi^*(K_{X_{a,b}}) + \bar{L} + \bar{L}_{11} + \bar{L}_{12} + \bar{L}_{21} +\bar{L}_{22} + 
\bar{L}_{13} + \bar{L}_{14} + \bar{L}_{23} + \bar{L}_{24} + 2\bar{E}_{3} + 2\bar{E}_{4} + 2\bar{E}_{5}.$
Here, for example, $\bar{L}=\pi(L)$. 
Note that $L$ is not integral and hence $\pi^*(\bar{L})=2L$ (Lemma \ref{insep}).
Combining these two equations and Lemma \ref{MIIDer},(2), (3), we have
$$\pi^*\phi^*K_{X_{a,b}}=0,$$
and hence $K_{X_{a,b}}$ is numerically trivial.  Since $b_2(Y^{D_{a,b}})=b_2(Y)=22$, we have
$$b_2(X_{a,b})= b_2(Y^{D_{a,b}}) - 12 = 10.$$
Thus $X_{a,b}$ is an Enriques surface.

The elliptic fibration $g:Y\to {\bf P}^1$ induces an elliptic fibration $f: X_{a,b} \to {\bf P}^1$
which has two singular fibers of type ${\rm I}_4$ and a singular fiber of type ${\rm III}$ consisting of the images of $F_1$ and  $F_2$.
Since the images of two smooth integral curves stated in Lemma \ref{MIIDer}, (3) are multiple fibers of the elliptic fibration, $X_{a,b}$ is classical.  

Recall that $Y$ contains 42 nodal curves. 
Except $E_1, E_2$, $L$ and the eleven integral nodal curves, the images of the remaining 28 nodal curves
are nodal curves on $X_{a,b}$.  The images of $E_1$, $E_2$ are rational curves with a cusp. 
It is not difficult to see that the configuration of these 28 nodal curves is given as in the Figure \ref{A7-2}.
Thus we conclude:

\begin{theorem}\label{mainMII} The surface $X_{a,b}$ is a classical
Enriques surface whose canonical cover
$\bar{Y}$ has eight nodes and one rational double point of type $D_4$.  The minimal resolution of $\bar{Y}$   is the supersingular 
$K3$ surface $Y$ with Artin invariant $1$.  The surface $X_{a,b}$ contains $28$ nodal curves as in the Figure {\rm \ref{A7-2}}.
\end{theorem}

\begin{theorem}\label{MII12}
There are twelve non-effective $(-2)$-divisors on $X_{a,b}$.  The dual graph of the $28$ nodal curves and these $12$ $(-2)$-vectors satisfies the condition in Proposition {\rm \ref{Vinberg}}.  
In particular the reflection subgroup generated by the 
reflections associated with these $40$ $(-2)$-vectors is of finite index in ${\rm O}({\rm Num}(X_{a,b}))$.
\end{theorem}
\begin{proof}
Among the 168 $(-4)$-vectors given in (\ref{168}), 
the desired ones are the images of the divisors perpendicular to the root lattice 
$D_4 \oplus A_1^{\oplus 8}$ generated by
the exceptional curves of the singularities of the canonical cover of $X$. 
Such divisors correspond to six point sets $\calS$ on ${\bf P}^2({\bf F}_4)$ such that $\calS$ contains $p_1, p_2$, does not
contain $p_3, p_4, p_5$ and each line $\ell_{ij}$ passes through 
one member in $\calS \setminus \{p_1, p_2\}$.
%Let $p_{ij}$ be the intersection of the lines $\ell_{1i}$ and $\ell_{2j}$.
We can easily see that there are exactly 12 such sets in general position.
Thus we have 12 $(-2)$-vectors $\{ r_i\}_{i=1}^{12}$ in ${\rm Num}(X_{a,b})$.  Each vector $r_j$ has the intersection multiplicity 1 with exactly 7 vectors among $\{r_i\}$ and 2 with the remaining 4 vectors.
The dual graph of $\{r_i\}$ has two types of maximal parabolic subdiagrams of type $\tilde{A}_1$ and 
$\tilde{A}_2$.
It follows that maximal parabolic subdiagrams of the dual graph of 40 vectors are of type
$$\tilde{A}_3+\tilde{A}_3 + \tilde{A}_1+\tilde{A}_1,\ \tilde{A}_5+\tilde{A}_2+\tilde{A}_1, \ \tilde{A}_7 +\tilde{A}_1,\
\tilde{A}_5 + \tilde{A}_2+ \tilde{A}_1, \ \tilde{A}_2+\tilde{A}_2+\tilde{A}_2+\tilde{A}_2$$
all of which have the maximal rank 8.  Note that there are two cases of type $\tilde{A}_5+\tilde{A}_2+\tilde{A}_1$.  In one of them, $\tilde{A}_2$ consists of three non-effective $(-2)$-vectors and
in another one, all vertices are represented by effective $(-2)$-vectors (see Remark \ref{RemarkMII}). 
\end{proof}

%By the same argument as in the case of type ${\rm MII}$ (see Proposition \ref{MIfibrations}), 
%we have the following Proposition.

\begin{proposition}\label{sing-fibers}
There exist exactly five types of elliptic fibrations on $X_{a,b}$ as follows$:$
$$({\rm I}_4, {\rm I}_4, {\rm III}),\ ({\rm I}_6, {\rm IV}, {\rm I}_2),\ ({\rm I}_8, {\rm III}),\ 
({\rm I}_6, 2{\rm III}),\ (2{\rm IV}, 2{\rm IV}, {\rm IV}).$$
All fibrations are special.
\end{proposition}
\begin{proof}
The set of elliptic fibrations on $Y$ up to ${\rm Aut}(Y)$ bijectively corresponds to
the set of primitive isotropic vectors in ${\rm NS}(Y)$ contained in the closure of the finite polyhedron
defined by 42 nodal curves and 168 $(-2)$-curves.  It follows that any elliptic 
fibration on $X_{a,b}$ corresponds to a primitive isotropic vector in ${\rm Num}(X_{a,b})$ contained in the closure of the finite polyhedron
defined by $40$ $(-2)$-vectors mentioned in Theorem \ref{MII12}. 
The latter one corresponds to a maximal parabolic
subdiagram of the dual graph of 40 vectors.  
The types of fibers (e.g. type ${\rm III}$ or type ${\rm I}_2$) 
are determined by the classification of
elliptic fibrations after Lemma \ref{possibleSing}. Their multiplicities follows from the 
construction (see the next example \ref{MIIfibrationExample}).
\end{proof}

\begin{example}\label{MIIfibrationExample}
We give examples of elliptic fibrations in Proposition \ref{sing-fibers} and their special bi-sections.
\begin{itemize}
\item[$(1)$] Take a point from $\{p_3,p_4,p_5\}$, for example, $p_3$.  
Let $\ell_3, \ell_3'$ be two lines on ${\bf P}^2({\bf F}_4)$ passing through $p_3$ (see 
Figure \ref{A7-1}).  Denote by $F$, $F'$ the image of $\ell_3$, $\ell_3'$ on $X=X_{a,b}$ respectively.  Then 
$F,F'$ are nodal curves in the right hand side of Figure \ref{A7-2} which form a singular fiber of type ${\rm III}$ of an elliptic fibration.  
Each of the two lines $\ell_3, \ell_3'$ passes four points
in ${\bf P}^2({\bf F}_4)\setminus \{p_1,..., p_5\}$.  
The images of the remaining eight points not lying on
$\ell_3, \ell_3'$ form two singular fibers of type ${\rm I}_4$ of the fibration.  There are two types of special bi-sections.
A nodal curve in the right hand side of Figure \ref{A7-2} not tangent to $F$ is one of them.  Another one is a bi-section tangent to $F$ or $F'$ at a simple point.  For example, the image of an ${\bf F}_4$-rational point on $\ell_3\setminus \{p_3\}$ is such a bi-section. Since a bi-section meets $F+F'$ with multiplicity 2, the fiber of type ${\rm III}$ is not multiple.  Thus we have an elliptic fibration of type $({\rm I}_4, {\rm I}_4, {\rm III})$.
\item[$(2)$] %Take three points from $p_1,...,p_5$, for example, $p_3, p_4, p_5$.
Consider three lines $\ell_3, \ell_4, \ell_5$ on ${\bf P}^2({\bf F}_4)$ such that 
$\ell_i$ passes $p_i$ $(i=3,4,5)$.  We assume that these three lines do not meet at a point.
Denote by $F_i$ the image of $\ell_i$ on $X$.  
Then $F_3, F_4, F_5$ form a singular fiber of type ${\rm IV}$.
Among the 7 points on ${\bf P}^2({\bf F}_4)$ not lying on $\ell, \ell_{3}, \ell_{4}, \ell_{5}$, 
there are six points whose images on $X$ form a singular fiber of type ${\rm I}_6$.  The remaining point
is a component of a singular fiber of type ${\rm I}_2$.
The image of an ${\bf F}_4$-rational point on $\ell_{3}\setminus\{p_3, \ell_3\cap \ell_{4}, \ell_3\cap \ell_{5}\}$ gives a special bi-section.  Since a bisection meets $F_3+F_4+F_5$ with multiplicity 2, 
the fiber of type ${\rm IV}$ is not multiple.  
Thus we have an elliptic fibration with singular fiber of type $({\rm I}_6, {\rm IV}, {\rm I}_2)$.
\item[$(3)$] Take a point from $p_3,p_4,p_5$, for example, $p_3$, and a line $\ell_3$ on ${\bf P}^2({\bf F}_4)$ passing through $p_3$.   
Among the 12 points on ${\bf P}^2({\bf F}_4) \setminus \{\ell, \ell_{3}\}$ there are eight points $q_1,...,q_8$ (not unique) whose images on $X$ form a singular fiber $F$ of type ${\rm I}_8$.  By Proposition \ref{sing-fibers}, this fibration has a singular fiber $F'$ of type ${\rm III}$ or 
$2{\rm III}$.  The image $L_3$ of $\ell_{3}$ on $X$ is a component of $F'$.
Denote by $L_3'$ the remaining component of $F'$.  Then $L_3$ and $L_3'$ are tangent at the point $p_0$
the image of the singular point of type $D_4$ on the canonical covering of $X$ (see Lemma \ref{possibleSing4}, (1)). 
Take a point $p$ of ${\bf P}^2({\bf F}_4) \setminus \{\ell,\ell_3, q_1,...,q_8\}$ and a point $q$ on 
$\ell_{3}\setminus \{p_3\}$.  Consider a line $\ell'$ passing through $p$ and $q$. Then $\ell'$ meets 
$\ell$ at $p_4$ or $p_5$ and passes through exactly one point among the above eight points.  This implies
that the image $s$ of $\ell'$ on $X$ is a bi-section of the fibration. Since $s$ meets both 
$L_3$ and $L_3'$ at
$p_0$ transversally, $F'$ is not multiple.
Thus we have an elliptic fibration with singular fiber of type $({\rm I}_8, {\rm III})$.
\item[$(4)$] Take three points from $p_1,...,p_5$, for example, $p_1, p_2, p_3$.  
Consider a line on ${\bf P}^2({\bf F}_4)$ passing through
$p_i$ $(i=1,2)$, for example, $\ell_{11}, \ell_{21}$.   
Consider the image of the point $\ell_{11}\cap \ell_{21}$ on $X$ and that of the line $\ell_3$ passing through
two points $\ell_{11}\cap \ell_{21}$ and $p_3$.  
These two nodal curves form a singular fiber of type ${\rm III}$ of an elliptic fibration.
The image of an ${\bf F}_4$-rational point on $\ell_{11}\setminus \{p_1, \ell_{11}\cap\ell_{21}\}$ is a bi-section.  Since this bi-section meets the fiber of type ${\rm III}$ with multiplicity 1, 
this fiber is multiple.  The images of six points not lying on $\ell, \ell_{11}, \ell_{21}, \ell_3$
form a singular fiber of type ${\rm I}_6$.   
Thus we have an elliptic fibration of type $({\rm I}_6, 2{\rm III})$.  There are two types of bi-sections.
The above one passes a singular point of the fiber of type ${\rm I}_6$.  Another one is tangent to a component of the fiber of type ${\rm I}_6$.  For example, the image of a line passing through the point 
$p_4$, but not through $\ell_{11}\cap \ell_{21}$ gives such a bi-section.  
Since it passes the canonical point which is the image of 
the rational double point of type $D_4$,
it is tangent to a component of the fiber of type ${\rm I}_6$ by Lemma \ref{nodal}.
\item[$(5)$] %Take three points from $p_1,...,p_5$, for example, $p_3, p_4, p_5$.
Consider three lines $\ell_3, \ell_4, \ell_5$ on ${\bf P}^2({\bf F}_4)$ such that 
$\ell_i$ passes through $p_i$ $(i=3,4,5)$.  We assume that these three lines meet at a point $p$ (compare this with the case $(2)$).   Let $\ell_j$ be the line passing through $p$ and $p_j$ $(j=1,2)$.
Denote by $L_i$ the image of $\ell_i$ on $X$ $(i=3,4,5)$.  
Then $L_3, L_4, L_5$ form a singular fiber $F$ of type ${\rm IV}$ of an elliptic fibration.
On the other hand, the images of the three ${\bf F}_4$-rational points lying on $\ell_{1}\setminus \{p_1, p\}$ form a singular fiber $F_1$ of type ${\rm IV}$ of the elliptic fibration.  Similary
the three ${\bf F}_4$-rational points on $\ell_{2}\setminus \{p_2, p\}$ gives a singular fiber $F_2$ of
type ${\rm IV}$ of the fibration.  The image of a point in ${\bf P}^2({\bf F}_4)\setminus \{\ell\cup \ell_{1}\cup \ell_{2}\}$ gives a bi-section.  Since this bi-section meets $F_1, F_2$ (resp. $F$) with
multiplicity 1 (resp. 2), $F_1, F_2$ are multiple and $F$ is not.
Thus we have an elliptic fibration of type $(2{\rm IV}, 2{\rm IV}, {\rm IV})$.  
\end{itemize}
\end{example}

\begin{remark}\label{RemarkMII}
We remark that the maximal parabolic subdiagrams in the proof of Theorem \ref{MII12} correspond to five types of elliptic fibrations in Proposition \ref{sing-fibers} and
that, in the first, fourth, or fifth case, a parabolic subdiagram of type $\tilde{A}_1$, $\tilde{A}_2$ or $\tilde{A}_2$ respectively contains a non-effective $(-2)$-vector.
\end{remark}

\begin{remark}
The symmetry group of the dual graph of $40$ $(-2)$-vectors is 
$(\mathfrak{S}_4 \times \mathfrak{S}_4)\cdot {\bf Z}/2{\bf Z}$ (see Figure \ref{A7-1}).  
This remarkable diagram of $(-2)$-vectors was first discovered by Shigeru Mukai (unpublished) 
in case of complex Enriques  surfaces.
\end{remark}

Recall that $\Aut(Y)$ is generated by $\PGL(3,{\bf F}_4)$, a switch and 
168 Cremona transformations, where $Y$ is the covering $K3$ surface of $X_{a,b}$ (Dolgachev and Kondo \cite{DK}).
Among these automorphisms, the subgroup $(\mathfrak{S}_4 \times \mathfrak{S}_4)\cdot {\bf Z}/2{\bf Z}$ and the twelve Cremona transformations associated with the twelve divisors stated in the proof of Theorem \ref{MII12} preserve the 12 nodal curves
$$L, L_{11},\ L_{12},\ L_{21},\ L_{22},\ L_{13},\ L_{14},\ L_{23},\ L_{24},\ E_{3},\ E_{4},\ E_{5}$$
contracted under the map $\phi$.

\begin{conjecture}
The subgroup $(\mathfrak{S}_4 \times \mathfrak{S}_4)\cdot {\bf Z}/2{\bf Z}$ and 
the twelve Cremona transformations descend to automorphisms of $X_{a,b}$.
\end{conjecture}

\begin{lemma}\label{bi-section3} 
There are nine canonical points on $X_{a,b}$.  One of them, denoted by $p_0$, corresponds to the $D_4$-singularity and the others correspond to $A_1$-singularities.  Let $f:X_{a,b}\to {\bf P}^1$ 
be an elliptic fibration.
\noindent
\begin{itemize}
\item[$(1)$] In case that $f$ is of type $({\rm I}_4, {\rm I}_4, {\rm III})$, the canonical point $p_0$ is the singular
point of the fiber of type ${\rm III}$.  For special bi-sections, the following two cases occur. 
Either a special bi-section passes through $p_0$ and is tangent to the other two singular fibers of type 
${\rm I}_4$ at a simple point, or it passes through a singular point of two fibers of type ${\rm I}_4$ and is tangent to the fiber of type ${\rm III}$ at a simple point.
\item[$(2)$] In case of $({\rm I}_6,  {\rm IV}, {\rm I}_2)$, the canonical point $p_0$ is the singular
point of the fiber of type ${\rm IV}$.  Any special bi-section passes through a singular point of 
the fiber of type ${\rm I}_6$ and a singular point of the fiber of type ${\rm I}_2$, and is tangent to 
the fiber of type ${\rm IV}$ at a simple point.
\item[$(3)$] In case of type $({\rm I}_8, {\rm III})$, the canonical point $p_0$ is the singular
point of the fiber of type ${\rm III}$.  Any special bi-section passes through $p_0$ and is tangent to
the fiber of type ${\rm I}_8$ at a simple point.
\item[$(4)$] In case of type $({\rm I}_6, 2{\rm III})$, the canonical point $p_0$ is a simple point of a component of the singular fiber of type $2{\rm III}$.  Two canonical points lie on the other component of the fiber of type $2{\rm III}$ both of which are simple points of the fiber.  
For special bi-sections, the following two cases occur.
A special bi-section passes through $p_0$ and is tangent to the fiber of type 
${\rm I}_6$ at a simple point.  Or a special bi-section
passes through a canonical point on the fiber of type ${\rm III}$ not $p_0$ and a singular point {\rm (}=canonical point{\rm )} on the fiber of type ${\rm I}_6$.
\item[$(5)$] In case of type $(2{\rm IV}, 2{\rm IV}, {\rm IV})$, the canonical point $p_0$ is the singular point of the non-multiple fiber of type ${\rm IV}$.  
The other two singular fibers contain four canonical points.
One of them is the singular point of the fiber and others are simple points of each component.
Any special bi-section passes through a canonical point on two multiple fibers and is tangent to 
the remaining fiber at a simple point. 
\end{itemize}
\end{lemma}
\begin{proof}
The existence of bi-sections of given types follows from Example \ref{MIIfibrationExample}.
In cases (2), (3), the elliptic fibration $g$ on $Y$ induced from $f$ has a section and
the Mordell-Weil group is a torsion group.  Any special bi-sections of $g$ is 
one of 28 nodal curves mentioned in Theorem \ref{mainMII}.  Thus we directly prove the assertion.

In the remaining cases, the proof is similar to that of Lemma \ref{bi-section1}.
Any special bi-section of $f$ is a section of $g$.  We know which components of $g$ are integral curves.  From this one can easily check the assertions.
\end{proof}

\section{Possible singularities and singular fibers}\label{sec4}

Let $X$ be an Enriques surface.  Assume that the canonical cover $\bar{\pi} :\bar{Y}\to X$ has
only rational double points and the minimal resolution $Y$ of $\bar{Y}$ is the supersingular $K3$
surface with the Artin invariant 1. 
In this section we determine the possibilities of the singularities of $\bar{Y}$ (Lemma \ref{possibleSing}) and
study elliptic fibrations on $X$ (Lemmas \ref{possibleSing3}, \ref{possibleSing4}, \ref{possibleSing5}).

%We first recall a known result on singularities of $\tilde{X}$.  

\begin{proposition}\label{EHS-sing}
Let $R$ be the lattice generated by exceptional curves of the minimal resolution $Y\to \bar{Y}$ of singularities.  
Then $R$ is one of the following$:$
$$A_1^{\oplus 12},\ A_1^{\oplus 8}\oplus D_4,\ A_1^{\oplus 4}\oplus D_4^{\oplus 2},\ A_1^{\oplus 6}\oplus D_6.$$
\end{proposition}
\begin{proof}
Denote by $\bar{R}$ the primitive sublattice in ${\rm Pic}(Y)$ containing $R$ of finite index.
Then $\bar{R}$ is the orthogonal complement of $\pi^*({\rm Pic}(X)) \cong E_{10}(2)$ in ${\rm Pic}(Y)$.
Since $\rho(Y)=22$, ${\rm rank}(R) =12$.  
By Ekedahl, Hyland, Shepherd-Barron \cite{EHS}, Lemma 6.5, $R$ is the direct sum of root lattices 
of type $A_1$, $D_{2n}$, $E_7$, $E_8$.  Note that these root lattices are 2-elementary, that is,
$R^*/R$ is a 2-elementary abelian group.  Since $\bar{R}$ is an over lattice of $R$, $\bar{R}$ is also 2-elementary (Nikulin \cite{N}, Proposition 1.4.1).  Assume that $R^*/R \cong ({\bf Z}/2{\bf Z})^a$ and
$\bar{R}^*/\bar{R} \cong ({\bf Z}/2{\bf Z})^{a'}$.  Then $a' \leq a$ (Nikulin \cite{N}, Proposition 1.4.1).     
Denote by $H$ the quotient group ${\rm Pic}(Y)/(E_{10}(2) \oplus \bar{R})$.
It follows from Nikulin \cite{N}, Proposition 1.5.1 that
$$2^{10+a'} = |{\rm det}(E_{10}(2) \oplus \bar{R})|= |{\rm det}({\rm Pic}(Y))| \cdot |H|^2 = 2^2 \cdot |H|^2.$$
Hence we have
$|H| = 2^{4 + a'/2}$.  Since $H$ is embedded into $\bar{R}^*/\bar{R}$ 
(Nikulin \cite{N}, Proposition 1.5.1),
we have $2^{4+a'/2} \leq 2^{a'}$, 
and hence $8 \leq a' \leq a$.  
Now the assertion follows from the classification of root lattices of rank 12.
\end{proof}

Let $f : X \to {\bf P}^1$ be an elliptic fibration on $X$ and denote by $g: Y \to {\bf P}^1$ the
induced elliptic fibration on $Y$.  It follows that $g$ is one of eight elliptic fibrations given in Theorem \ref{ElkiesSchutt2}.  
%Recall that all fibration have a section.

\begin{lemma}\label{possibleSing}
The contribution of a fiber of $g$ to the rational double points on $\bar{Y}$ is as follows.
\begin{itemize}
\item[$(1)$] On a singular fiber of type ${\rm I}_{2n}$ there exist $n$ disjoint components contracting to $n$ rational double points of type $A_1$.
\item[$(2)$] On a singular fiber of type ${\rm I}_1^*$, there are two possibilities$:$ the four simple components of the fiber are contracted to
four rational double points of type $A_1$, or four components forming a dual graph of type $D_4$ are contracted to a rational double point of type $D_4$.
\item[$(3)$] On a singular fiber of type ${\rm I}_3^*$, there are two possibilities$:$ two simple components are contracted to two rational double points of type $A_1$ and another four components forming a dual graph of type $D_4$ are contracted to a rational double point of type $D_4$, or six components forming a dual graph of type $D_6$ are contracted to a rational double point of type $D_6$
\item[$(4)$] On a singular fiber of type ${\rm IV}^*$ there are two possibilities$:$ the four disjoint components are contracted to four rational double points of type $A_1$, or four components forming a dual graph of type $D_4$ are contracted to a rational double point of type $D_4$.
\end{itemize}
\end{lemma}
\begin{proof}
First note that each component of singular fibers of type ${\rm I}_{2n}$, ${\rm I}_1^*$, ${\rm I}_3^*$, 
${\rm IV}^*$ meets transversally with other components.
This implies that integral curves in these fibers form a disjoint union of nodal curves. 
Also note that possible singularities are 12 $A_1$-singularities, 8 $A_1$- and a $D_4$-singularities, 4 $A_1$- and 2 $D_4$-singularities, or 6 $A_1$- and a $D_6$-singularities (Proposition \ref{EHS-sing}).

The first assertion for ${\rm I}_{2n}$ ($n\geq 2$) follows from Lemma \ref{nodal}.
In case of ${\rm I}_{2}$, one component of the fiber is integral and the other is not (if both are integral
or non integral, this contradicts to Lemma \ref{insep}).  Hence the assertion (1) follows.

In case of a fiber of type ${\rm I}_1^*$, there are at most four disjoint components.  If there are four integral curves, then they are four simple components and correspond to four rational double points of type $A_1$.  If the number of integral curves is three, 
then they are two simple components and a component with multiplicity 2 (otherwise, the image of the fiber
to $X$ is not a configuration of Kodaira's type).
Together with the component meeting with three integral curves, they form the exceptional curves of
a rational double point of type $D_4$.

In case of a fiber of type ${\rm I}_3^*$, there are at most five disjoint components.  
If there are five integral curves, then they are four simple components and a component with multiplicity 2.  There exists a unique component meeting three integral curves.  They form the exceptional curves of
a rational double point of type $D_4$.  The remaining two integral curves correspond to two rational double points of type $A_1$.
If the number of integral curves is four, then they are two simple components and two components with multiplicity 2 (otherwise, the image of the fiber
to $X$ is not a configuration of Kodaira's type). 
Together with two components meeting at least two integral curves, they form the exceptional curves of
a rational double point of type $D_6$.

In case of a fiber of type ${\rm IV}^*$, there are at most four disjoint components.  If the number of integral curves is three, then they are three components of the fiber with multiplicity 2 (otherwise, the image of the fiber to $X$ is not a configuration of Kodaira's type).
Together with the component meeting these three curves, they form the exceptional curves of
a rational double point of type $D_4$. If the number of integral curves is four, then they are three
simple components of the fiber and the component with multiplicity 3. These four components correspond to
four $A_1$-singularities.
\end{proof}

By the proof of Lemma \ref{possibleSing}, we can determine the image of each singular fiber to $X$.
Thus we have the following types of the elliptic fibration
$f : X \to {\bf P}^1$ corresponding to $g: Y \to {\bf P}^1$ (see Theorem \ref{ElkiesSchutt2}).

\bigskip
\halign{\hfil\tt#\hfil&&\quad#\hfil\cr
& Type of $g$: & $({\rm I}_6, {\rm I}_6, {\rm I}_6, {\rm I}_6)$, & $({\rm I}_8, {\rm I}_8, {\rm I}_1^*)$, & $({\rm I}_{10}, {\rm I}_{10}, {\rm I}_2, {\rm I}_2)$, &  $({\rm I}_{12}, {\rm I}_3^*)$ \cr
\noalign{\smallskip}
& Type of $f$: & $({\rm I}_3, {\rm I}_3, {\rm I}_3, {\rm I}_3)$, & $({\rm I}_4, {\rm I}_4, {\rm III})$, & 
$({\rm I}_{5}, {\rm I}_{5}, {\rm I}_1, {\rm I}_1)$, & $({\rm I}_{6}, {\rm III})$\cr}

\bigskip

\halign{\hfil\tt#\hfil&&\quad#\hfil\cr
& Type of $g$: & $({\rm I}_{12}, {\rm I}_4, {\rm IV}^*)$, & $({\rm IV}^*, {\rm IV}^*, {\rm IV}^*)$, & $({\rm I}_{16}, {\rm I}_1^*)$, &  $({\rm I}_{18}, {\rm I}_2, {\rm I}_2, {\rm I}_2)$ \cr
\noalign{\smallskip}
& Type of $f$: & $({\rm I}_{6}, {\rm I}_2, {\rm IV})$, & $({\rm IV}, {\rm IV}, {\rm IV})$, & $({\rm I}_{8}, {\rm III})$, & $({\rm I}_{9}, {\rm I}_1, {\rm I}_1, {\rm I}_1)$ \cr}
\bigskip

%\begin{lemma}\label{possibleSing2}
%The types of singularities on $\tilde{X}$ are three possibilities:  twelve rational double points of type $A_1$, one rational double point of type $D_4$ and eight rational double points of type $A_1$, or one rational double point of type $D_6$ and six rational double points of type $A_1$. 
%\end{lemma}
The following three lemmas easily follow from Lemma \ref{possibleSing} and its proof.

\begin{lemma}\label{possibleSing3}
Assume that $\bar{Y}$ has a rational double point of type
$D_6$ and six rational double points of type $A_1$.  Then $X$ has only one type 
$({\rm I}_{6}, {\rm III})$
of singular fibers of elliptic fibrations, and $\bar{Y}$ has six rational double points of type $A_1$ 
over the six singular points of the fiber of type ${\rm I}_6$ and a rational double point of type $D_6$ 
over the singular point of the fiber of type ${\rm III}$.
\end{lemma}
%\begin{proof}
%This is obvious. 
%\end{proof}
\begin{lemma}\label{possibleSing4-2}
Assume that $\bar{Y}$ has two rational double points of type
$D_4$ and four rational double points of type $A_1$.  Then $X$ has only one type
$({\rm IV}, {\rm IV}, {\rm IV})$ of singular fibers of elliptic fibrations, and 
$\bar{Y}$ has rational double points of type $D_4$ over the singular points of two fibers of type ${\rm IV}$ and four rational double points of type $A_1$ 
over four points on the remaining  singular fiber $F$ of type ${\rm IV}$. One of the four points is the 
singular point of the fiber $F$ and the remaining three points consist of a point on each component
of $F$.  
\end{lemma}

\begin{lemma}\label{possibleSing4}
Assume that $\bar{Y}$ has a rational double point of type
$D_4$ and eight rational double points of type $A_1$.  Then $X$ can have five types
of singular fibers of elliptic fibrations as follows$:$ 
$$({\rm I}_4, {\rm I}_4, {\rm III}),\ ({\rm I}_{6}, {\rm III}), \ ({\rm I}_{6}, {\rm I}_2, {\rm IV}),\ 
({\rm IV}, {\rm IV}, {\rm IV}),\ ({\rm I}_{8}, {\rm III}).$$
\begin{itemize}
\item[$(1)$] In cases of $({\rm I}_4, {\rm I}_4, {\rm III}),\ ({\rm I}_{6}, {\rm I}_2, {\rm IV}),\ ({\rm I}_{8}, {\rm III})$, $\bar{Y}$ has eight rational double points of type $A_1$ 
over the eight singular points of singular fibers of type ${\rm I}_n$ and a rational double point of type $D_4$ 
over the singular point of the fiber of type ${\rm III}$ or type ${\rm IV}$.  
\item[$(2)$] In case of $({\rm I}_{6}, {\rm III})$, 
$\bar{Y}$ has eight rational double points of type $A_1$ 
over the six singular points of the fiber of type ${\rm I}_6$ and two points on a component of the
singular fiber of type ${\rm III}$ and a rational double point of type $D_4$ 
over a point of the other component of the fiber of type ${\rm III}$.  
\item[$(3)$] In case of $({\rm IV}, {\rm IV}, {\rm IV})$, $\bar{Y}$ has a rational double point of type $D_4$ over the singular point of a fiber of type ${\rm IV}$ and eight rational double points of type $A_1$ 
over the eight points on the remaining two singular fibers of type ${\rm IV}$. Two of the eight points are
two singular points of the two fibers and the remaining six points consist of a point on each component
of the two singular fibers.  
\end{itemize}  
\end{lemma}
%\begin{proof}
%This is obvious. 
%\end{proof}

\begin{lemma}\label{possibleSing5}
Assume that $\bar{Y}$ has twelve rational double points of type $A_1$ and elliptic fibrations are special.  Then $X$ can have six types
of singular fibers of elliptic fibrations as follows$:$ 
$$({\rm I}_3, {\rm I}_3, {\rm I}_3, {\rm I}_3), \ ({\rm I}_4, {\rm I}_4, {\rm III}),\ 
({\rm I}_{5}, {\rm I}_5, {\rm I}_1, {\rm I}_1), \ ({\rm I}_{6}, {\rm I}_2, {\rm IV}),\ 
({\rm I}_{8}, {\rm III}), \ ({\rm I}_9, {\rm I}_1, {\rm I}_1, {\rm I}_1).$$
\begin{itemize}
\item[$(1)$] In cases of $({\rm I}_3, {\rm I}_3, {\rm I}_3, {\rm I}_3), ({\rm I}_{5}, {\rm I}_5, {\rm I}_1, {\rm I}_1), ({\rm I}_9, {\rm I}_1, {\rm I}_1, {\rm I}_1)$, $\bar{Y}$ has twelve rational double points of type $A_1$  over the twelve singular points of the fibers of type ${\rm I}_n$.  
\item[$(2)$] In case of $({\rm I}_4, {\rm I}_4, {\rm III}), ({\rm I}_{8}, {\rm III})$, $\bar{Y}$ has eight rational double points of type $A_1$ 
over the eight singular points of the fibers of type ${\rm I}_n$ and four rational double points of type $A_1$ over four points on the singular fiber of type ${\rm III}$. Each component of the fiber of type ${\rm III}$ contains two of the four points. 
\item[$(3)$] In case of $({\rm I}_{6}, {\rm I}_2, {\rm IV})$, $\bar{Y}$ has eight rational double points of type $A_1$ over the eight singular points of the fibers of type ${\rm I}_n$ and four rational double points of type $A_1$ over the fiber of type ${\rm IV}$.  One of the four points is the singular point of the 
fiber and the remaining three points consist of a point on each component of the fiber.  
\end{itemize}  
\end{lemma}
\begin{proof}
The only non trivial thing is non-existence of the case $({\rm IV}, {\rm IV}, {\rm IV})$.
In this case, three singular fibers of $g: Y\to {\bf P}^1$ are of type ${\rm IV}^*$ and all simple components are integral.
Recall that $g$ has a section $s$ whose image on $X$ is a special bi-section of $f$ (we assume that $f$ is special).  This implies that
all three singular fibers of $f$ are multiple which is a contradiction (also $s$ passes three canonical points which is impossible (Lemma \ref{nodal})).
%(Also $S$ meets three integral curves and hence contracts to a point on $X$.  This destroies the fibration structure of $p$).
\end{proof}

\section{Special bi-sections of a special elliptic fibration}\label{sec5}

%We use the same notation as in the previous section.
In this section we study possibilities of special bi-sections of a special elliptic fibration
$f: X \to {\bf P}^1$ on an Enriques surface $X$. We assume that the canonical cover $\bar{Y}$ of $X$ has only rational double points and its minimal nonsingular model $Y$ is the supersingular $K3$ surface with the Artin invariant 1.
Let $s$ be a special bi-section.
In the following lemmas \ref{possiblebisection1}, \ref{possiblebisection4}, 
\ref{possiblebisection5}, \ref{possiblebisection2}, \ref{possiblebisection3},  \ref{possiblebisection6}, 
we assume that the canonical cover $\bar{Y}$ has only rational double points of type $A_1$.  

\begin{lemma}\label{possiblebisection1} 
In case that $f$ has singular fibers of type $({\rm I}_5, {\rm I}_5, {\rm I}_1, {\rm I}_1)$,
the following three cases occur$:$
\begin{itemize}
\item[$(1)$] $s$ passes through a singular point of two fibers of type ${\rm I}_5$.
\item[$(2)$] $s$ passes through a singular point of two fibers of type ${\rm I}_1$.
\item[$(3)$] $s$ passes a singular point of a fiber of type ${\rm I}_5$ and that of a fiber of type ${\rm I}_1$.
\end{itemize}
\end{lemma}
\begin{proof}
Since $\bar{Y}$ has only rational double points of type $A_1$, any special bi-section passes through two
canonical points (Lemma \ref{nodal}).  Hence the assertion is obvious. 
\end{proof}

By the same proof as that of Lemma \ref{possiblebisection1}, we have the following two lemmas 
\ref{possiblebisection4}, \ref{possiblebisection5}.

\begin{lemma}\label{possiblebisection4} 
In case that $f$ has singular fibers of type $({\rm I}_9, {\rm I}_1, {\rm I}_1, {\rm I}_1)$, 
the following two cases occur$:$
\begin{itemize}
\item[$(1)$] $s$ passes through a singular point of the fiber of type ${\rm I}_9$.
\item[$(2)$] $s$ passes through a singular point of two fibers of type ${\rm I}_1$.
\end{itemize}
\end{lemma}

\begin{lemma}\label{possiblebisection5} 
In case that $f$ has singular fibers of type $({\rm I}_3, {\rm I}_3, {\rm I}_3, {\rm I}_3)$, $s$ passes through a singular point of two fibers of type ${\rm I}_3$.
\end{lemma}

\begin{lemma}\label{possiblebisection2} 
In case that $f$ has singular fibers of type $({\rm I}_6, {\rm IV}, {\rm I}_2)$, the fiber of type
${\rm IV}$ is multiple and 
the following two cases occur$:$
\begin{itemize}
\item[$(1)$] $s$ passes through a singular point of the fiber of type ${\rm I}_6$.
\item[$(2)$] $s$ passes through a singular point of the fiber of type ${\rm I}_2$.
\end{itemize}
\end{lemma}
\begin{proof}
Since the pre-image of the fiber of type ${\rm IV}$ on $Y$ is of type ${\rm IV}^*$ and
three simple components of the fiber of type ${\rm IV}^*$ are integral (see the proof of 
Lemma \ref{possibleSing}, (4)), 
$s$ passes through exactly one canonical point on the fiber of type ${\rm IV}$ and hence this fiber is multiple.
Since $s$ passes through another canonical point, either the assertion (1) or (2) holds.
\end{proof}

\begin{lemma}\label{possiblebisection3} 
In case that $f$ has singular fibers of type $({\rm I}_8, {\rm III})$, the fiber of type ${\rm III}$ is
multiple and $s$ passes through a singular point of the fiber of type ${\rm I}_8$.
\end{lemma}
\begin{proof}
Since there are exactly four canonical points on the fiber of type ${\rm III}$ which are the images of
the four simple components of the fiber of type ${\rm I}_1^*$  (see the proof of 
Lemma \ref{possibleSing}, (2)), the fiber of type ${\rm III}$ is multiple.
The bi-section $s$ passes through another canonical point and hence the remaining assertion follows.
\end{proof}

\begin{lemma}\label{possiblebisection6} 
In case that $f$ has singular fibers of type $({\rm I}_4, {\rm I}_4, {\rm III})$, 
the fiber of type ${\rm III}$ is multiple. The bi-section $s$ passes through a singular point of 
a fiber of type ${\rm I}_4$ and is tangent to a component of the other fiber of type ${\rm I}_4$.
\end{lemma}
\begin{proof}
The proof is the same as that of Lemma \ref{possiblebisection3}. 
\end{proof}

Next, in the following lemmas \ref{possiblebisection7}, 
\ref{possiblebisection8}, \ref{possiblebisection9}, \ref{possiblebisection10}, 
\ref{possiblebisection11}, we 
assume that $\bar{Y}$ has a rational double point of type $D_4$ and 8 rational double points of type
$A_1$.  
%Let $f: X \to {\bf P}^1$ be a special elliptic fibration and denote by $s$ a special bi-section.
Denote by $p_0$ the canonical point on $X$ which is the image of the rational double point of type $D_4$ 
on $\bar{Y}$.

\begin{lemma}\label{possiblebisection7} 
In case that $f$ has singular fibers of type $({\rm I}_8, {\rm III})$, 
the fiber of type ${\rm III}$ is not multiple and its singular point is $p_0$. The bi-section $s$ passes 
through the singular point of the fiber of type ${\rm III}$ and is tangent to a component of the fiber of type ${\rm I}_8$.
\end{lemma}
\begin{proof}
Since the two components of the fiber $F$ of type ${\rm III}$ correspond to two simple components of
the fiber of type ${\rm I}_1^*$ (see the proof of 
Lemma \ref{possibleSing}, (2)), $s$ is tangent to a component of $F$ or passes through the singular 
point of $F$.  Hence $F$ is not multiple.  If $s$ is tangent to $F$, then $s$ passes through two singular points
of the fiber of type ${\rm I}_8$ which is impossible because $s$ is a bi-section.  Therefore $s$ passes
through the singular point of $F$ and is tangent to a component of the fiber of type ${\rm I}_8$.
\end{proof}

\begin{lemma}\label{possiblebisection8} 
In case that $f$ has singular fibers of type $({\rm I}_6, {\rm III})$, 
the fiber of type ${\rm III}$ is multiple, $p_0$ is a simple point of a component of this fiber, and 
the following two cases occur$:$
\begin{itemize}
\item[$(1)$] $s$ passes through a singular point of the fiber of type ${\rm I}_6$.
\item[$(2)$] $s$ is tangent to a component of the fiber of type ${\rm I}_6$.
\end{itemize}
\end{lemma}
\begin{proof}
The pullback of the fiber $F$ of type ${\rm III}$ to $Y$ is of type ${\rm I}_3^*$, and the image of
the cycle of type $D_4$ is nothing but $p_0$ (see the proof of 
Lemma \ref{possibleSing}, (3)).  Since four simple components of the fiber of type ${\rm I}_3^*$ are integral, $s$ meets $F$ at a simple point on $F$ 
transversally. Therefore $F$ is multiple.  If $s$ passes through $p_0$, then $s$ is tangent to a component
of the fiber $F'$ of type ${\rm I}_6$.  If $s$ passes through a canonical point of $F$ not equal $p_0$, then $s$ passes through a singular point of $F'$.
\end{proof}

\begin{lemma}\label{possiblebisection9} 
In case that $f$ has singular fibers of type $({\rm IV}, {\rm IV}, {\rm IV})$, 
two fibers of type ${\rm IV}$ are multiple and $p_0$ is the singular point of the non multiple fiber of type ${\rm IV}$. The bi-section $s$ is tangent to a component of the non-multiple fiber of type ${\rm IV}$.
\end{lemma}
\begin{proof}
Since $p_0$ is the image of the cycle of type $D_4$ on a fiber of type ${\rm IV}^*$, $p_0$ is the
singular point of a fiber of type ${\rm IV}$ (see the proof of 
Lemma \ref{possibleSing}, (4)).  Since $s$ passes through a canonical and simple point of the 
other two singular fibers of type ${\rm IV}$, these two singular fibers are multiple.  The remaining
assertion is obvious.
\end{proof}

\begin{lemma}\label{possiblebisection10} 
In case that $f$ has singular fibers of type $({\rm I}_4, {\rm I}_4, {\rm III})$, 
the fiber of type ${\rm III}$ is not multiple, $p_0$ is its singular point and  
the following two cases occur$:$
\begin{itemize}
\item[$(1)$] $s$ passes through a singular point of two fibers of type ${\rm I}_4$.
\item[$(2)$] $s$ passes through $p_0$ and is tangent to the fibers of type ${\rm I}_4$ at a simple point.
\end{itemize}
\end{lemma}
\begin{proof}
The proof of non-multipleness of the fiber of type ${\rm III}$ is the same as that of Lemma
\ref{possiblebisection7}.  The remaining assertions are obvious.
\end{proof}

\begin{lemma}\label{possiblebisection11} 
In case that $f$ has singular fibers of type $({\rm I}_6, {\rm IV}, {\rm I}_2)$, 
the fiber of type ${\rm IV}$ is not multiple and $p_0$ is its singular point.
The bi-section $s$ passes through a singular point of the fiber of type 
${\rm I}_6$ and a singular point of the fiber of type ${\rm I}_2$.
\end{lemma}
\begin{proof}
The proof of the first assertion is similar to that of Lemma \ref{possiblebisection9}.  The remaining
assertion is obvious. 
\end{proof}

Finally we consider the case that the canonical double cover 
$\bar{Y}$ has two rational double points of type $D_4$ or a rational double point of type $D_6$.

\begin{lemma}\label{possiblebisection13}
The canonical cover $\bar{Y}$ has neither a rational double points of type $\tilde{D}_6$ nor two rational double points of type $D_4$.
\end{lemma}
\begin{proof}
(due to the referee)

\noindent
If $\bar{Y}$ has a rational double point of type $D_6$, then $f$ has singular fibers of type
$({\rm I}_6, {\rm III})$ (Lemma \ref{possibleSing3}).
The induced fibration $g:Y\to {\bf P}^1$ has singular fibers of type $({\rm I}_{12}, {\rm I}_3^*)$, and 
the special bi-section $s$ induces a section $\tilde{s}$ of $g$.
By Lemma \ref{nodal}, $\tilde{s}$ has exactly two points (including an infinitely near point) contracted
during the blow-down $Y \to \bar{Y}$.  On the other hand, 
by Lemma \ref{possibleSing}, (3),  
$\tilde{s}$ has either three points contracted during $Y\to \bar{Y}$ or no such points 
on the fiber of type ${\rm I}_3^*$, and by Lemma \ref{possibleSing}, (1), 
$\tilde{s}$ has either one point contracted during $Y\to \bar{Y}$ or no such points 
on the fiber of type ${\rm I}_{12}$.  This is a contradiction.

If $\bar{Y}$ has two rational double points of type $D_4$, then $f$ has singular fibers of type
$({\rm IV}, {\rm IV}, {\rm IV})$ (Lemma \ref{possibleSing4-2}).  The induced fibration $g:Y\to {\bf P}^1$
has three singular fibers of type ${\rm IV}^*$ and has a section $\tilde{s}$.
It follows from Lemma \ref{possibleSing}, (4) that $\tilde{s}$ has exactly one point contracted during
$Y \to \bar{Y}$.  This contradicts Lemma \ref{nodal}.
\end{proof}

\begin{remark}
Ekedahl, Hyland and Shepherd-Barron \cite[Corollary 6.16]{EHS} showed that 
the canonical cover $\bar{Y}$ has no two rational double points of type $D_4$ in general setting without the assumption of the existence of a special bi-section.  Very recently Matsumoto \cite{Ma} studied more details of possible singularities on the canonical coverings of Enriques surfaces in characteristic 2.
\end{remark}

\section{Classification}\label{sec6}

In the following, $X$ is an Enriques surface and $Y$ is the minimal resolution of
the canonical cover $\bar{Y}$ of $X$.  We assume that $Y$ is the supersingular $K3$ surface with the 
Artin invariant 1.
If $X$ has no nodal curves, then any elliptic fibration on $X$ has only irreducible fibers.
The induced fibration on $Y$ is one of the list of Theorem \ref{ElkiesSchutt2} which is impossible.
Thus $X$ contains a nodal curve, and hence $X$ has a special elliptic fibration (Proposition \ref{Cossec}).
We fix a special elliptic fibration $f:X \to {\bf P}^1$ with a special bi-section $s$.

\begin{theorem}\label{Main1}

\noindent
\begin{itemize}
\item[$(a)$] Assume that $\bar{Y}$ has only rational double points of type $A_1$.  Then $X$ has the dual graph of twenty $(-2)$-vectors of type ${\rm VII}$ or forty $(-2)$-vectors of type ${\rm MI}$. 
\item[$(b)$]
Assume that $\bar{Y}$ has a rational double point of type $D_4$.  Then $X$ has the dual graph of
forty $(-2)$-vectors of type ${\rm MII}$. 
\end{itemize}
\end{theorem}
\begin{proof}
By Lemmas \ref{possibleSing4}, \ref{possibleSing5}, \ref{possiblebisection13}, 
we know possible singular fibers and the positions of
canonical points on the fibers.  Moreover, for each special elliptic fibration, we know 
the configuration of singular fibers and special bi-sections (Lemma \ref{possiblebisection1} -- \ref{possiblebisection11}).  Fortunately, if we take any possible special bi-section $s$, it coincides with the unique one of Examples of type ${\rm VII}$,
type ${\rm MI}$ or type ${\rm MII}$.  For example, in Lemma \ref{possiblebisection1}, 
the cases (1) and (2) correspond to the Example of type ${\rm VII}$ (Lemma \ref{bi-section2}, (1))
and 
the case  (3) corresponds to the Example of type ${\rm MI}$ (Lemma \ref{bi-section1}, (1)).
In each case, the pullback of the fibration gives an elliptic fibration on the supersingular
$K3$ surface $Y$ which is unique up to isomorphisms (Theorem \ref{ElkiesSchutt}).  Therefore
we have the remaining nodal curves on $Y$ as sections or multi-sections in each case, and hence 
we obtain the remaining nodal curves on $X$ whose dual graph is the same as that of the corresponding Example.
In case of type ${\rm VII}$, the dual graph of nodal curves is already determined.
In case of type ${\rm MI}$ or ${\rm MII}$, the remaining 10 or 12 $(-2)$-vectors are determined
by the obtained 30 or 28 nodal curves.  Thus we have the dual graph of $20$, $40$, or $40$ vertices, respectively, which satisfies the condition in Proposition \ref{Vinberg}.
\end{proof}
\noindent
Thus we have the main theorem in this paper.

\begin{theorem}\label{Main2}
There exist exactly three types of Enriques surfaces such that the minimal resolutions of the canonical double covers of these Enriques surfaces are the supersingular $K3$ surface with the Artin invariant $1$.
\end{theorem}

\begin{remark}\label{EHSB-unique}
By the result of Ekedahl, Hyland and Shepherd-Barron \cite[Theorem 3.21]{EHS}, the canonical cover of
each of the examples of Enriques surfaces of type MI, MII, VII has exactly 2-dimensional regular derivations.
Thus it follows that our examples give all Enriques surfaces such that the minimal resolutions of the canonical double covers of these Enriques surfaces are the supersingular $K3$ surface with the Artin invariant $1$.
\end{remark}


\begin{thebibliography}{99}
\bibitem{A} M. Artin, \textit{Supersingular $K3$ surfaces,} Ann. Sci. \'Ecole Norm. Sup., 4 (1974), 543--567.
\bibitem{BM2}E. Bombieri, D. Mumford, \textit{Enriques' classification of surfaces in char. $p$,} III, Invent. Math., {\bf 35} (1976), 197--232.
\bibitem{Bo} R.\ Borcherds, \textit{Automorphism groups of
Lorentzian lattices}, J. Algebra {\bf 111} (1987), 133--153.
\bibitem{Cossec} F.\ Cossec, \textit{On the Picard group of Enriques surfaces,} Math. Ann., {\bf 271} (1985), 577--600.
\bibitem{CD}F.\ Cossec and I.\ Dolgachev, \textit{Enriques surfaces {\rm I}}, Progr. Math., vol. {\bf 76}, 1989, Birkh\"auser.
\bibitem{D}I.\ Dolgachev, \textit{On automorphisms of Enriques surfaces,} Invent. Math., {\bf 76} (1984), 163--177.
\bibitem{DK}I. Dolgachev, S. Kond\=o, \textit{A supersingular $K3$ surface in characteristic 2 and Leech lattice,} IMRN 2003 (2003), 1--23.
%\bibitem{DL} I. Dolgachev, C. Liedtke, \textit{Enriques surfaces} I, manuscript in 2018, July.
\bibitem{EkS}T. Ekedahl, N. I. Shepherd-Barron, \textit{On exceptional Enriques surfaces,} arXiv:math/0405510.
\bibitem{EHS}T. Ekedahl, J. M. E. Hyland, N. I. Shepherd-Barron, \textit{Moduli and periods of simply connected Enriques surfaces,} arXiv:1210.0342.
\bibitem{ES}N. Elkies, M. Sch\"utt, \textit{Genus $1$ fibrations on the supersingular $K3$ surface in characteristic $2$ with Artin invariant $1$,} Asian J. Math., {\bf 19} (2015), 555--581.
\bibitem{Esselmann} F.\ Esselmann, \textit{Ueber die maximale Dimension
von Lorenz-Gittern mit coendlicher Spiegelungsgruppe}, J. Number Theory {\bf 61}(1996), 103--144.
%\bibitem{Fano} G. Fano, \textit{Superficie  algebriche di genere zero  bigenere uno e loro casi particulari}, Rend. Circ. Mat. Palernmo, {\bf 29} (1910), 98--118.
%\bibitem{Ho} E. Horikawa, \textit{Periods of Enriques surfaces. {\rm II}}, Math.Ann., {\bf 235} (1978), 217--246.
\bibitem{Ill} L.\ Illusie, \textit{Complexe de de Rham-Witt et cohomologie cristalline}, Ann. Sci. \'Ecole Norm. Sup., {\bf 12} (1979), 501--661.
%\bibitem{I} H.\ Ito, \newblock{On extremal elliptic surfaces in characteristic $2$ and $3$}, \newblock{Hiroshima Math. J., {\bf 32} (2002), 179--188.}
\bibitem{KK1}T. Katsura, S. Kond\=o, \textit{A $1$-dimensional family of Enriques surfaces in characteristic $2$ covered by the supersingular $K3$ surface with Artin invariant $1$,} Pure Appl. Math. Q., {\bf 11} (2015), 683--709.
\bibitem{KK2}T. Katsura, S. Kond\=o, \textit{Enriques surfaces in characteristic $2$ with a finite group of automorphisms,} J. Algebraic Geometry {\bf 27} (2018), 173--202.
\bibitem{KKM}T. Katsura, S. Kond\=o, G. Martin, \textit{Classification of Enriques surfaces with finite automorphism group in characteristic 2,} to appear in Algebraic Geometry; arXiv:1703.09609v2.
\bibitem{KT}T. Katsura, Y. Takeda, \textit{Quotients of abelian and hyperelliptic surfaces by rational vector fields,} J. Algebra, {\bf 124} (1989), 472--492.
\bibitem{Ko} S. Kond\=o, \textit{Enriques surfaces with finite automorphism groups}, Japanese J. Math., {\bf 12} (1986), 191--282.
\bibitem{KS} S. Kond\=o, I. Shimada, \textit{On a certain duality of N\'eron-Severi lattices of supersingular $K3$ surfaces,} Algebraic Geometry {\bf 3} (2014), 311--333.
\bibitem{Lang} W.\ Lang, \textit{On Enriques surfaces in characteristic $p$},  I, II, Math. Ann., {\bf 265} (1983), 45--65; ibid., {\bf 281} (1988), 671--685.
%\bibitem{L1} W.\ Lang, \newblock{Extremal rational elliptic surfaces in characteristic p. I. Beauville surfaces}, \newblock{Math. Z., {\bf 207} (1991), 429--438.}
%\bibitem{L2} W.\ Lang, \newblock{Extremal rational elliptic surfaces in characteristic p. II. Surfaces with three or fewer singular fibres}, \newblock{Ark. Mat., {\bf 32} (1994), 423--448.}
\bibitem{L}C.\ Liedtke, \textit{Arithmetic moduli and liftings of Enriques surfaces,} J. Reine Angew. Math., {\bf 706} (2015), 35--65.
\bibitem{Ma} Y.\ Matsumoto, \textit{Canonical coverings of Enriques surfaces in characteristic 2}, arXiv:1812.06914v2.
\bibitem{MO} S. Mukai, H. Ohashi, \textit{Finite groups of automorphisms of Enriques surfaces and the Mathieu group}, arXiv:1410.7535v1.
\bibitem{N} V.\ Nikulin, \textit{Integral symmetric bilinear forms and its applications}, 
Math. USSR Izv., {\bf 14} (1980), 103--167.
%\bibitem{N} V.\ Nikulin, \textit{On a description of the automorphism groups of Enriques surfaces,} Soviet Math. Dokl., {\bf 30} (1984), 282--285.
\bibitem{Og} A.\ Ogus, \textit{Supersingular $K3$ crystals,} Asterisque {\bf 84} (1979), 3--86.
%\bibitem{Og2} A.\ Ogus, \textit{A crystalline Torelli theorem for supersingular $K3$ surfaces,} in {\em Arithmetic and geometry}, Vol. II, Progr. Math. {\bf 36}, 361--394. Birkh\"auser Boston, Boston, MA, 1983.
\bibitem{Oh1} H. Ohashi, \textit{On the number of Enriques quotients of a $K3$ surface}, Publ. RIMS, Kyoto Univ., {\bf 43} (2007), 181--200.
\bibitem{Oh2} H. Ohashi, \textit{Enriques surfaces covered by Jacobian Kummer surfaces}, Nagoya Math. J., {\bf 195} (2009), 165--186.
\bibitem{RS}A. N. Rudakov, I. R. Shafarevich, \textit{Inseparable morphisms of algebraic surfaces,} Izv. Akad. Nauk SSSR Ser. Mat., {\bf 40} (1976), 1269--1307.
\bibitem{RS2}A.\ N.\ Rudakov and I.\ R.\ Shafarevich, \newblock{Supersingular surfaces of type $K3$ over fields of characteristic 2,} \newblock{Izv. Akad. Nauk SSSR Ser. Mat., 42 (1978), 848--869.}
%\bibitem{RS3}A.\ N.\ Rudakov and I.\ R.\ Shafarevich, 
%\newblock{Surfaces of type K3 over fields of finite characteristic,}
%\newblock{J. Soviet Math. 22 (1983), 1476--1533.}
\bibitem{V} E. B. Vinberg, \textit{Some arithmetic discrete groups in Lobachevskii spaces}, in "Discrete subgroups of Lie groups and applications to Moduli", Tata-Oxford (1975), 323--348.
\end{thebibliography}
\end{document}